\documentclass[]{siamart171218} 
\usepackage{graphicx}        

\usepackage{color}
\usepackage{amsmath}
\usepackage{amssymb}
\usepackage{mathrsfs}
\usepackage{mathtools}  
\usepackage[font=small,labelfont=bf]{caption}
\usepackage{hyperref}
\usepackage{color}
\usepackage{enumitem}
\usepackage{chngcntr}
\usepackage{bm} 
\usepackage{algorithmic,algorithm}
\usepackage[hang,flushmargin]{footmisc} 
\usepackage{xr}

\numberwithin{theorem}{section}
\numberwithin{equation}{section}

\renewcommand{\cal}[1]{\mathcal{#1}}

\newcommand{\n}{\mathbb{N}}

\newcommand{\zp}{\mathbb{Z}_+}
\newcommand{\mmag}[1]{\left|#1\right|}
\newcommand{\s}{\mathcal{S}}
\newcommand{\tws}{2^\mathcal{S}}

\newcommand{\norm}[1]{\left|\left|{#1}\right|\right|}

\newcommand{\Pb}{\mathbb{P}}

\newcommand{\Pbl}[1]{\mathbb{P}\left(#1\right)} 
\newcommand{\Ebl}[1]{\mathbb{E}\left[#1\right]} 

\headers{The exit time finite state projection scheme}{Juan Kuntz, Philipp Thomas, Guy-Bart Stan, and Mauricio Barahona}
\title{The exit time finite state projection scheme: \\ 
bounding exit distributions and occupation measures of continuous-time Markov chains\thanks{The first author was supported by a BBSRC PhD Studentship (BB/F017510/1). PT was supported through a Fellowship of The Royal Commission for the Exhibition of 1851. GBS acknowledges support by an EPSRC Fellowship for Growth (EP/M002187/1) and the EU H2020-FETOPEN-2016-2017 project 766840 (COSY-BIO). MB acknowledges support from EPSRC grant EP/N014529/1 supporting the EPSRC Centre for Mathematics of Precision Healthcare.}}

\author{
  Juan Kuntz\thanks{Department of Mathematics and Department of Bioengineering, Imperial College London, London, SW7 2AZ, UK
    (\email{juan.kuntz08@imperial.ac.uk}).
    }
  \and
  Philipp Thomas\thanks{Department of Mathematics, Imperial College London, London, SW7 2AZ, UK 
  (\email{p.thomas@imperial.ac.uk}).
  }
  \and
  Guy-Bart Stan\thanks{Department of Bioengineering, Imperial College London, London, SW7 2AZ, UK
    (\email{Co-corresponding author: g.stan@imperial.ac.uk}).}
  \and
  Mauricio Barahona\thanks{Department of Mathematics, Imperial College London, London, SW7 2AZ, UK
    (\email{Co-corresponding author: m.barahona@imperial.ac.uk}).}
}

\begin{document}

\maketitle
\begin{abstract}
We introduce the exit time finite state projection (ETFSP) scheme, a truncation-based method that yields approximations to the exit distribution and occupation measure associated with the time of exit from a domain (i.e., the time of first passage to the complement of the domain) 
of time-homogeneous continuous-time Markov chains.
We prove that: (i) the computed approximations bound the measures from below; (ii) the total variation distances between the approximations and the measures decrease monotonically as states are added to the truncation; and (iii) the scheme converges, 
in the sense that, as the truncation tends to the entire state space, the total variation distances tend to zero. 
Furthermore, we give a computable bound on the total variation distance between the exit distribution and its approximation, 
and we delineate the cases in which the bound is sharp. 
We also revisit the related finite state projection scheme and give a comprehensive account of its theoretical properties.
We demonstrate the use of the ETFSP scheme by applying it to two biological examples: the computation of the first passage time associated with the expression of a gene, and the fixation times of competing species subject to demographic noise.
\end{abstract}



\section{Introduction}

The time of exit of a continuous-time Markov chain from a domain (or \emph{exit time} for short) is the time at which the chain leaves the domain for the first time. The exit time is also known as the~\emph{first passage time} or, alternatively, the \emph{hitting time} of the complement of the domain.
Two measures are associated with an exit event: the \emph{exit distribution}, which describes when and where the chain exits the domain, and the \emph{occupation measure}, which describes which states the chain visits before exiting and at what times they are visited. 
These two measures can be expressed in terms of the \emph{time-varying law} (i.e, the state space distribution of the chain as a function of time)
of an \emph{auxiliary chain} that is identical to the original chain except that every state outside of the domain is turned into an absorbing state~\cite{Syski1992,Kuntzthe}. 

There exists a rich literature on exit times, especially in physics and biomathematics~\cite{iyer2016,metzler2014,Redner2001}. Recently, there has been renewed interest in exit times of 
continuous-time Markov chains with discrete state space \cite{assaf2017,ghusinga2017,valleriani2014,schnoerr2017}, such as those we study in this paper. 
While the exit problem from a small finite domain is tractable \cite{gillespie1991,grima2016}, the exit problem from an infinite or large domain can only be solved in special 
cases~\cite{drummond2010,ghusinga2017,munsky2009}. Various approximation schemes have been developed to address this issue~\cite{barzel2008,Dandach2010,hinch2005,schnoerr2017}. However, most of them do not provide bounds or error estimates on their accuracy.

The popular \emph{finite state projection} (FSP) scheme~\cite{Munsky2006} yields lower bounds on the time-varying law of the continuous-time chain of interest. The scheme chooses a finite truncation of the state space and solves an associated system of linear ordinary differential equations (ODEs) indexed by the states contained in the truncation. 
Inbuilt in the procedure is a computable upper bound on the total variation distance between the lower bounds obtained and the time-varying law. 
However, 
the FSP does not provide information about the exit from a domain.

To fill this gap, we introduce the \emph{exit time finite state projection} (ETFSP) scheme (Sec.~\ref{etfspsintro}), that involves applying an FSP-like scheme to the auxiliary chain with an absorbing complement mentioned above.
We show that the scheme yields lower bounds on the exit distribution and the occupation measure associated with an exit time. For the exit distribution, we explain how to compute a bound on the error of its approximation.
Theorem~\ref{etfspthrm} delineates the theoretical properties of ETFSP showing that: (i) the error bound is sharp if and only if the exit event occurs with probability one, and (ii) the scheme converges in total variation to the exit distribution and occupation measure as the truncation approaches the entire state space.

A secondary contribution is Theorem \ref{fspthrm}, which gathers the theoretical properties of the FSP scheme. In particular, we show that the error bound of the FSP scheme is sharp if and only if the chain is non-explosive, in which case the error bound can indeed be made arbitrarily small by including enough states in the truncation. In the explosive case, the error bound remains non-zero, as observed in \cite{MacNamara2007}, and is limited by the probability of explosion.

The final contribution of this paper is a new proof of an old theorem: Theorem~\ref{charactt} expresses the exit distribution in terms of the time-varying law of the aforementioned auxiliary chain. Versions of this theorem pepper the literature  (e.g., \cite{Feller1971,kampen2007,Neuts1994,Melamed1984,Redner2001,Syski1992}). Our variant relaxes the non-explosive and deterministic initial condition assumptions in \cite{Syski1992}, and adds the analogous result for the occupation measure.

\paragraph*{Related literature}
To the best of our knowledge, the ideas underpinning the FSP and ETFSP schemes were first delineated in the 1980's queuing literature (see \cite{Gross1984,Melamed1984,Melamed1984a} and references therein) centred around randomisation techniques for continuous-time chains with bounded rate matrices.
Recently, 
schemes based on the FSP have been used to bound the cumulative density function of the exit time of stochastic reaction networks, a subclass of the continuous-time chains that we consider here~\cite{Cao2013,Cao2016,Dandach2010}. Given that the ETFSP scheme bounds not only the cumulative density functions but also the corresponding densities, our results imply the convergence of those other schemes as a special case.  

\subsection{The ETFSP scheme: statement of the problem and main result}
We briefly define our problem setting, introduce the ETFSP scheme, 
and state our main result (Theorem \ref{etfspthrm}) detailing the theoretical properties of the scheme.

\subsubsection{Problem definition} 
Let $X:=\{X_t:0\leq t< T_\infty\}$ be a minimal time-homogeneous continuous-time Markov chain on a probability triplet $(\Omega,\cal{F},\Pb)$  with countable state space $\s$, stable and conservative rate matrix $Q:=(q(x,y))_{x,y\in\s}$,  explosion time $T_\infty$, and initial distribution
$\gamma(x):=\Pbl{\{X_0=x\}}, \, \forall x\in\s.$
%

We single out a subset $\cal{D}$ of the state space $\s$ and refer to it as the \emph{domain}. The \emph{exit time} $\tau$ from the domain 
is the time when the chain first leaves $\cal{D}$:
\begin{equation} 
\label{eq:hitthec}
\tau(\omega):=\inf\{t\in[0,T_\infty(\omega)): X_t(\omega)\not\in\cal{D}\},\quad\forall \omega\in\Omega,
\end{equation}
with the convention that the infimum of the empty set is infinity: $\inf \{\emptyset \} =\infty$. 

The exit distribution $\mu$ and occupation measure $\nu$ associated with the exit time are defined as:  
\begin{align}
\mu([a,b),x)&:=\Pbl{\{\tau\in[a,b),X_\tau= x\}}&\forall 0\leq a<b<\infty,\, \forall x\in\s,\label{eq:edisdef}\\
\nu([a,b),x)&:=\Ebl{\int_{a\wedge \tau\wedge T_\infty }^{b\wedge \tau\wedge T_\infty } 1_{x}(X_t)dt}& \forall 0\leq a<b<\infty,\, \forall x\in\s,\label{eq:eoccdef} 
\end{align}
where $c \wedge d = \min(c,d), \, c,d \in \mathbb{R}$, 
and $1_x$ denotes the indicator function of state $x$: $1_x(y)=1$ if $y=x$ and $0$ otherwise.  

For each state $x$, the measures $\mu(dt,x)$ and $\nu(dt,x)$ have densities $\mu(t,x)$ and $\nu(t,x)$ with respect to the Lebesgue measure. (We distinguish a measure from its density by writing $dt$ or $t$ in its argument.)
For small $h>0$, the distribution $t \mapsto \mu(t,x)$ is a function such that $\mu(t,x) \, h$ is the probability that the chain first exits the domain via state $x$ during the time interval $[t,t+h]$.  Similarly (and assuming non-explosivity for the chain), $\nu(t,x)$ is the average fraction of the interval $[t,t+h]$ that the chain spends in state $x$ before exiting the domain. Formally, the relationship between the exit distribution and occupation measure and their densities is: 
\begin{align}\label{eq:md1}
\mu([a,b),x)&=1_{\cal{D}^c}(x) \, \gamma(x) \, 1_0(a)+\int_a^b\mu(t,x)dt&\forall 0\leq a<b<\infty,\, \forall x\in\s,\\
\label{eq:md11}\nu([a,b),x)&=\int_a^b\nu(t,x)dt&\forall 0\leq a<b<\infty,\, \forall x\in\s,
\end{align}
where $\cal{D}^c$ denotes the complement of the domain 
and the term $1_{\cal{D}^c}(x) \, \gamma(x) \, 1_0(a)$
captures the event that the chain is started outside of the domain.

In this paper, we introduce the \textit{exit time finite state projection} scheme to approximate the exit distribution and occupation measure in a systematic manner. ETFSP yields
approximations of the densities $\mu(t,x)$ and $\nu(t,x)$, 
and, consequently, of their marginals, including the distribution of the exit time, $\tau$, and of the exit location, $X_\tau$.

\subsubsection{The exit time finite state projection (ETFSP) scheme}\label{etfspsintro} 
The numerical scheme consists of the following steps:
\begin{enumerate}
\item Choose a finite subset, or \emph{truncation}, $\s_r$ of the state space $\s$ and a \emph{final computation time}, $t_f^r\in[0,\infty)$.
\item Solve the set of $\mmag{\s_r}$ linear ODEs:
\begin{align} 
\label{eq:nrdef}
\dot \nu^r (t,x) &=\sum_{y\in\cal{D}_r}\nu^r(t,y) \, q(y,x), \\ \nu^r(0,x) & =\gamma(x),\quad \forall x\in\cal{D}_r, \nonumber\\
\dot \mu^r (t,x) &=\sum_{y\in\cal{D}_r}\nu^r(t,y)\left(\sum_{z\in\cal{D}_r}q(y,z)q(z,x)\right),\label{eq:mrdef}\\
\mu^r(0,x)&=\sum_{y\in\cal{D}_r}\gamma(y) \, q(y,x),\quad\forall x\in\s_r\cap\cal{D}^c,\nonumber
%
\end{align}
over the time interval $[0,t^r_f]$, where $\cal{D}_r$ denotes the \emph{truncated domain} $\cal{D}\cap\s_r$.
\item Pad $\nu^r$ and $\mu^r$ with zeros:
\begin{align}\label{eq:mrdef2}\begin{array}{lll}\nu^r(t,x):=0  &\text{if}\quad x\not\in\cal{D}_r\quad \text{or}\quad  t> t^r_f,\\
 \mu^r(t,x):=0& \text{if}\quad x\not\in\s_r\cap\cal{D}^c\quad \text{or}\quad t> t^r_f.\end{array}\end{align}
\end{enumerate}
The approximations 
of the measures $\mu(dt,x)$ and $\nu(dt,x)$ are defined as:
\begin{align}\label{eq:md2}
\mu^r([a,b),x)&:=1_{\cal{D}^c  \cap\s_r}(x) \gamma(x) 1_0(a)+\int_{a}^b\mu^r(t,x)dt\enskip\forall 0\leq a<b<\infty,\, \forall x\in\s,
\\ 
\nu^r([a,b),x)&:=\int_a^b\nu^r(t,x)dt\qquad\qquad\qquad\qquad\quad\quad\,\,\forall 0\leq a<b<\infty,\, \forall x\in\s.\label{eq:md22}\end{align}
\subsubsection{Theoretical characterisation of the ETFSP scheme}\label{etfsps_characterisation} Our main result is Theorem~\ref{etfspthrm}, which summarises the theoretical properties of the scheme. (For its proof, see Sec.~\ref{etfsp}.)
We show that $\mu^r$ and $\nu^r$ do not just approximate the exit distribution $\mu$ and occupation measure $\nu$, but bound them from below. We give simple expressions for the mass of the approximations and their errors in terms of the exit time $\tau$, the final computation time $t^r_f$, and exit time from the truncation $\s_r$,
\begin{equation}\label{eq:taurexit}\tau_r:=\inf\{0\leq t<T_\infty:X_t\not\in\s_r\}\qquad\forall r\in\n.\end{equation}
To quantify the \emph{approximation errors} $\norm{\mu-\mu^r}$ and $\norm{\nu-\nu^r}$, we use the total variation norm
\begin{align}
\norm{\rho}:=\sup\{\mmag{\rho(\cal{A})}:\cal{A}\in\cal{G}\}
\label{eq:norm_def}
\end{align}
on the measures.
We give easy-to-compute bounds for the approximation errors, 
and we show that the bounds are sharp if the chain exits the domain almost surely (a property that can be verified using Foster-Lyapunov criteria~\cite{Menshikov2014,Kuntzthe}). Lastly, we prove that the approximation errors and their bounds decrease monotonically as we increase the truncation $\s_r$, and that the errors tend to zero as $\s_r$ tends to $\s$.

\begin{theorem}[The exit time finite state projection scheme]\label{etfspthrm} 
Consider a minimal time-homogeneous continuous-time Markov chain with countable state space $\s$, stable and conservative rate matrix $Q:=(q(x,y))_{x,y\in\s}$,  explosion time $T_\infty$, initial distribution $\gamma:=(\gamma(x))_{x\in\s}$. Suppose that the initial distribution satisfies
\begin{equation}\label{eq:qhatint}
\sum_{x\in\cal{D}}\gamma(x)\mmag{q(x,x)}<\infty\end{equation}
for a given domain $\cal{D} \subseteq \s$ and let $\mu$ and $\nu$ denote, respectively, the exit distribution and occupation measure associated with the exit time $\tau$ from $\cal{D}$.
Let  $\{\s_r\}_{r\in\n}$ be an increasing sequence of finite sets contained in $\s$, $\{t^r_f\}_{r\in\n}$ an increasing sequence of non-negative final computation times, and $\{\mu^r\}_{r\in\n}$ and $\{\nu^r\}_{r\in\n}$ the sequences of ETFSP approximations of the exit distribution and occupation measure, respectively, defined by~\eqref{eq:nrdef}--\eqref{eq:mrdef2}. Then the following properties hold:
\begin{enumerate}[label=(\roman*)]
\item(Increasing sequence of lower bounds) \label{th:etfsp_i} 
\begin{align*}\mu^0(t,x)\leq \mu^1(t,x)\leq \dots \leq \mu(t,x)\qquad\forall x\in\s,\quad t\in[0,\infty),\\
\nu^0(t,x)\leq \nu^1(t,x)\leq \dots \leq \nu(t,x)\qquad\forall x\in\s,\quad t\in[0,\infty).\end{align*}
\item(Mass of the approximations)\label{th:etfsp_ii} The mass of $\mu^r$ is the probability that the chain exits the domain no later than exiting the truncation or the final time, i.e.,
$$\mu^r([0,\infty),\s)=\Pbl{\{\tau\leq t^r_f\wedge\tau_r\}},\qquad\forall r\in\n.$$
The mass of $\nu^r$ is:
$$\nu^r([0,\infty),\s)=\Ebl{(\tau\wedge t^r_f) 1_{\{\tau\leq \tau_r\}}},\qquad\forall r\in\n.$$
\item(Computable error bounds) \label{th:etfsp_iii}
\label{thr:error_bound}
For any $r\in\n$,
\begin{align}
\norm{\mu-\mu^r} =& \, \Pbl{\{\tau<\infty\}}-\Pbl{\{\tau\leq t^r_f\wedge\tau_r\}}\label{eq:merr}\\
\leq& \, 1-\left(\gamma(\cal{D}^c\cap \s_r)+\sum_{x\in\cal{D}^c\cap \s_r}\int_0^{t^r_f} \mu^r(t,x)dt\right)=:\varepsilon_r,\nonumber\\
\norm{\nu-\nu^r}=&\Ebl{\tau\wedge T_\infty}-\Ebl{(\tau\wedge t^r_f)1_{\{\tau\leq\tau_r\}}}\label{eq:nerr}\\
\leq& \, \Ebl{\tau}-\sum_{x\in\cal{D}_r}\int_0^{t^r_f}\nu^r(t,x)dt=:\varepsilon^\nu_r.\notag
\end{align}
Equality holds for~\eqref{eq:merr} if and only if $\Pbl{\{\tau<\infty\}}=1$, i.e., when the chain exits the domain with probability one.
\item(Monotonicity of the error and of the error bound) \label{th:etfsp_iv} The approximation errors and their upper bounds are decreasing in $r$:
\begin{align}
\norm{\mu-\mu^r} &\geq ||\mu-\mu^{r+1}|| \quad \text{and} \quad
\varepsilon_r \geq \varepsilon_{r+1}, \quad \forall r\in\n  \\
\norm{\nu-\nu^r} &\geq ||\nu-\nu^{r+1}|| \quad \text{and} \quad
\varepsilon_r^\nu \geq \varepsilon_{r+1}^\nu, \quad \forall r\in\n.\label{eq:nerrmon}
\end{align}
%
%
\item(Convergence of bounds) \label{th:etfsp_v} If $\cup_{r}\s_r=\s$ and $t^r_f\to\infty$ as $r\to\infty$, the approximation $\mu^r$ converges in total variation to the exit distribution $\mu$:
$$\lim_{r\to\infty}\norm{\mu-\mu^r}=0.$$
Consequently, it follows from~\ref{thr:error_bound} that:
$$\lim_{r\to\infty}\varepsilon_r=0 \iff \Pbl{\{\tau<\infty\}}=1.$$
If $\Ebl{\tau\wedge T_\infty}<\infty$, the approximation $\nu^r$ converges in total variation to the occupation measure $\nu$: $$\lim_{r\to\infty}\norm{\nu-\nu^r}=0.$$
%
\end{enumerate}
\end{theorem}

We refer to the upper bound $\varepsilon_r$ defined in~\eqref{eq:merr} as the \emph{error bound} of the scheme because it bounds the approximation error of $\mu^r$. Note that the error bound is easily calculated from $\mu^r$, hence assessing the quality of the approximation requires no extra effort. 
The bound $\varepsilon^\nu_r$ for the occupation measure $\nu^r$ is harder to evaluate because the mean exit time, $\Ebl{\tau}$, is unknown in general. However, an upper bound on $\Ebl{\tau}$ can be obtained through additional computations beyond the scope of this paper (Sec.~\ref{conclu}). 

Condition~\eqref{eq:qhatint} is a mild technical assumption (e.g., it is satisfied if the chain is initialised deterministically) made to simplify the exposition by ensuring that the density of the exit distribution is finite at time zero: $\mu(0,x)<\infty$ for all $x\in\s$. 

\subsubsection{Paper structure} The remainder of the paper is structured as follows. In Sec.~\ref{prelim}, we formally define the chain and give several preliminary lemmas required in the subsequent proofs. Specifically, we review the forward equations and we provide proofs for theoretical properties of the original FSP (Sec.~\ref{fspsec}), and we give the analytical characterisation of the exit distribution and occupation measure and the marginals of these measures (Sec.~\ref{fspsec}). To ease the reading of the paper, we have relegated the technical proofs relevant to Sec.~\ref{prelim} to the Supplementary Material. 
Sec.~\ref{etfsp} contains the proof of Theorem~\ref{etfspthrm}. 
In Sec.~\ref{applications}, we apply the ETFSP scheme to two biologically motivated examples. We conclude by discussing possible implementations and extensions of the ETFSP scheme in Sec.~\ref{conclu}.

\section{Preliminaries}\label{prelim}  The starting point in our definition of a continuous-time chain is a stable and conservative rate matrix $Q:=(q(x,y))_{x,y\in\s}$, that is, a matrix of real numbers indexed by the countable state space $\cal{S}$ satisfying 
\begin{equation}\label{eq:qmatrix}q(x,y)\geq0\quad\forall x\neq y,\qquad q(x,x)=-\sum_{y\neq x}q(x,y)>-\infty,\quad\forall x\in\cal{S}.\end{equation}
Whenever we write ``a rate matrix $Q$'' in this paper, we mean ``a stable and conservative rate matrix $Q$''. We construct our Markov chain $X$ recursively by running the \emph{Gillespie Algorithm} \cite{Feller1940,Kendall1950,Gillespie1976} (see Appendix A in the Supplementary Material). In particular, the algorithm returns the \emph{jump times} $\{T_n\}_{n\in\n}$ at which transitions occur and the sequence $Y:=\{Y_n\}_{n\in\n}$ of states visited by the chain; both of these are defined on the same probability space $(\Omega,\cal{F},\Pb)$. The sequence $Y$ is itself a discrete-time Markov chain known as the \emph{jump chain} (or \emph{embedded chain}) and its one-step matrix is
\begin{equation}\label{eq:jumpmatrix}\pi(x,y):=\left\{\begin{array}{ll} \left (1_x(y)-1 \right) \, q(x,y)/q(x,x)&\text{if }q(x,x)\neq0\\  1_x(y)&\text{otherwise}\end{array}\right.,\qquad\forall x,y\in\s.\end{equation}
The sample paths $t\mapsto X_t(\omega)$ of the continuous-time chain $X$ are defined by
\begin{equation}\label{eq:cpathdef}X_t(\omega):=Y_n(\omega)\qquad\forall t\in [ T_n(\omega),T_{n+1}(\omega)),\quad \omega\in\Omega.\end{equation}
These paths are defined only up until the \emph{explosion time}
$$T_{\infty}(\omega):=\lim_{n\to \infty}T_n(\omega)\qquad\forall \omega\in\Omega.$$
The limit exists because $\{T_n(\omega)\}_{n\in\n}$ is an increasing sequence for each $\omega\in\Omega$. In other words, $X_t(\omega)$ is defined only for pairs $(t,\omega)$ such that $t<T_{\infty}(\omega)$. The reason behind the name ``explosion time'' given to $T_\infty$ is that, by this moment in time, the chain has left every finite subset of the state space. In particular, let $\s_0\subseteq \s_1\subseteq\dots$ be an increasing sequence of finite subsets (or \emph{truncations}) of $\s$ such that $\cup_{r}\s_r=\s$ and $\tau_r$ be the time \eqref{eq:taurexit} that the chain $X$ first exits $\s_r$. That our truncations form an increasing sequence implies that $\{\tau_r\}_{r\in\n}$ is an increasing sequence of random variables and the limit $\lim_{r\to\infty}\tau_r(\omega)$ exists for each $\omega\in\Omega$. 
\begin{lemma}[Lem.~2.18 of \cite{Kuntzthe}]\label{tautin} If $\{\s_r\}_{r\in\n}$ is an increasing sequence of finite sets such that $\cup_{r}\s_r=\s$, then $\tau_r$ tends to $T_\infty$ almost surely.
\end{lemma}

The limiting random variable $\lim_{r\to\infty}\tau_r$ is the point in time by which the chain has left each of the truncations in the sequence $\{\s_r\}_{r\in\n}$. The above tells us that $\lim_{r\to\infty}\tau_r$ is (almost surely) equal to $T_\infty$ \emph{regardless} of the particular sequence of finite truncations $\{\s_r\}_{r\in\n}$ in its definition. For this reason,  we interpret $T_\infty$ as the point in time that the chain leaves the state space, or, in other words, \emph{explodes}.

We now give two technical lemmas we will use throughout the paper. The first delineates the simple relationship between the exit time
\begin{equation}\label{eq:sigma}\sigma:=\inf\{n\in\n:Y_n\not\in\cal{D}\}\end{equation}
of the jump chain $Y$ and the exit time $\tau$ of $X$ (defined in \eqref{eq:hitthec}).
\begin{lemma}[Lem.~2.27 of \cite{Kuntzthe}]\label{hitdc} If $\tau$ and $\sigma$ are as in \eqref{eq:hitthec} and \eqref{eq:sigma}, then
$$\tau(\omega)=\left\{\begin{array}{ll}T_{\sigma(\omega)}(\omega)&\text{if }\sigma(\omega)<\infty\\\infty&\text{if }\sigma(\omega)=\infty\end{array}\right.\qquad\forall \omega\in\Omega.$$
\end{lemma}

The other lemma allows us to build \emph{auxiliary chains} that will be key in the proofs in this paper.
\begin{lemma}\label{samechainsh}Suppose that a second rate matrix $\bar{Q}$  coincides with  $Q$ on $\cal{D}$:
$$q(x,y)=\bar{q}(x,y),\qquad\forall x\in\cal{D},\quad y\in\s.$$
There exists a chain $\bar{X}:=\{\bar{X}_t\}_{t\geq0}$ also defined on $(\Omega,\cal{F},\Pb)$ with rate matrix $\bar{Q}$, jump times $\{\bar{T}_n\}_{n\in\n}$, jump chain $\bar{Y}:=\{\bar{Y}_n\}_{n\in\n}$, explosion time $\bar{T}_\infty$, and exit times
$$\bar{\sigma}:=\inf\{n\in\n:\bar{Y}_n\not\in\cal{D}\},\qquad \bar{\tau}:=\inf\{t\in[0,\bar{T}_\infty):\bar{X}_t\not\in\cal{D}\},$$
such that $X$ and $\bar{X}$ exit the domain at the same time:
\begin{equation}\label{eq:sametime}\sigma(\omega)=\bar{\sigma}(\omega),\quad \tau(\omega)=\bar{\tau}(\omega), \quad \tau(\omega)\wedge T_\infty(\omega)=\bar{\tau}(\omega)\wedge\bar{T}_\infty(\omega),\quad \forall\omega\in\Omega;\end{equation}
and that $X$ and $\bar{X}$ are identical up to (and including) this instant:
\begin{equation}\label{eq:sameproc}X_t(\omega)=\bar{X}_t(\omega), \qquad\forall  (t,\omega)\in[0,\infty)\times\Omega:t\leq\tau (\omega)\wedge T_\infty(\omega).\end{equation}
In particular, until the moment of exit, the jump chain and jump times of both chains are identical:
\begin{equation}\label{eq:sameyt}Y_n(\omega)=\bar{Y}_n(\omega),\qquad T_n(\omega)=\bar{T}_n(\omega),\qquad\forall (n,\omega)\in\n\times\Omega: n \leq \sigma(\omega).\end{equation}
\end{lemma}
\begin{proof}
See Appendix A in the Supplementary Material.
\end{proof}

\subsection{The time-varying law of the chain and the FSP scheme}\label{fspsec}

The time-varying law of the chain 
\begin{equation}\label{eq:lawdef}p_t(x):=\Pbl{\{X_t=x,t<T_\infty\}}\qquad\forall x\in\s,\end{equation}
satisfies $\mmag{\s}$ linear ordinary differential equations known as \emph{Kolmogorov's forward equations} (or \emph{the chemical master equation} or, simply, \emph{the forward equations}).
\begin{theorem}[Kolmogorov's forward equations, Cor.~2.21 of \cite{Kuntzthe}]\label{CME} Suppose that the diagonal of the rate matrix is $\gamma$-integrable:
\begin{equation}\label{eq:qdiaint}\Ebl{\mmag{q(X_0,X_0)}}=\sum_{x\in\s}\gamma(x)\mmag{q(x,x)}<\infty.\end{equation}
For each $x\in\s$, $t\mapsto p_t(x)$ is a continuously differentiable function on $[0,\infty)$. Furthermore, the time-varying law $p_t:=\{p_t(x)\}_{x\in\s}$ is the minimal non-negative solution of the equations
\begin{equation}\label{eq:laweqs}\dot{p}_t(x)=\sum_{y\in\s}p_t(y)q(y,x),\qquad p_0(x)=\gamma(x),\qquad \forall x\in\s,\quad t\in[0,\infty).\end{equation}
%
%
%
\end{theorem}

In the above, by ``minimal non-negative solution''  we mean that if $k_t$ is any other non-negative ($k_t(x)\geq 0$ for each $x\in\s$ and $t\in[0,\infty)$) differentiable function satisfying \eqref{eq:laweqs}, then $k_t(x)\geq p_t(x)$ for each $x\in\s$ and $t\geq0$ (if the chain is explosive, then the equations can have multiple solutions, see \cite{Chung1967,Freedman1983,Rogers2000a}). Except for a few special cases, no analytical expressions for this minimal solution are known. If $\s$ is infinite, or finite but large, direct numerical computation of this solution is not possible either. Instead, we can use the popular finite state projection (FSP) algorithm \cite{Munsky2006}: a numerical scheme that yields a set of lower bounds $p_t^r:=\{p_t^r(x)\}_{x\in\s}$ on the chain's time-varying law $p_t:=\{p_t(x)\}_{x\in\cal{S}}$. We identify these bounds with the measure on $(\s,\tws)$ defined by $p^r_t(A)=\sum_{x\in A}p^r_t(x)$ for all $A\subseteq \s$, where $\tws$ denotes the power set of $\s$. The FSP scheme consists of: choosing a (finite) truncation $\s_r$ of the state space $\s$; solving numerically the set of $\mmag{\s_r}$ linear ODEs
\begin{equation}\label{eq:fspode}
\dot{p}^r_t(x)=\sum_{y\in\s_r}p^r_t(y)q(y,x),\qquad p^r_0(x)=\gamma(x),\qquad \forall x\in\s_r,  
\end{equation}
over the time interval $[0,t]$; and padding $p_t^r$ with zeros: $p_t^r(x):=0$ for all $x\not\in\s_r$. 

We collect various useful properties of the FSP scheme in Theorem~\ref{fspthrm} below. Most of these properties can be found elsewhere: $\ref{th:fsp_i}$ and $\ref{th:fsp_v}$ are shown in Prop.~2.14 of~\cite{Anderson1991} (however, there is a small mistake therein, see~\cite{Chen1996}); $\ref{th:fsp_iv}$ and the bound in $\ref{th:fsp_iii}$ are proven in \cite{Munsky2006}. Although $\ref{th:fsp_ii}$ is mentioned in \cite{Dinh2016,Munsky2007}, we have not encountered a proof elsewhere. 
Similarly, the explicit expression of the error (i.e., the total variation distance between $p_t$ and its approximation) in $\ref{th:fsp_iii}$ 
and the necessary and sufficient condition for the bound to be sharp appear to be new. 
\begin{theorem}[The finite state projection scheme]\label{fspthrm} 
Let  $\{\s_r\}_{r\in\n}$ be an increasing sequence of finite sets contained in $\s$, $\tau_r$ the exit time from the truncation $\s_r$, 
and $\{p^r_t\}_{r\in\n}$ the sequence of FSP approximations defined by~\eqref{eq:fspode}. 
Then the following properties hold:
\begin{enumerate}[label=(\roman*)]
\item(Increasing sequence of lower bounds) \label{th:fsp_i}
%
%
$$p_t^0(x)\leq p_t^1(x)\leq \dots\leq p_t(x),\qquad\forall x\in\s\quad  t\geq0.$$
\item(Mass of the approximation)  \label{th:fsp_ii}
The mass of the approximation is the probability that the chain has not yet exited the truncation:
$$p_t^r(\s)=p_t^r(\s_r)=\Pbl{\{t<\tau_r\}},\qquad\forall t\geq0.$$
\item(Computable error bound)  \label{th:fsp_iii}
For any $r\in\n$,
$$\norm{p_t-p_t^r}=\Pbl{\{t<T_\infty\}}-\Pbl{\{t<\tau_{r}\}}\leq 1-p_t^r(\s_r),\qquad\forall t\geq0,$$
and equality holds if and only if $\Pbl{\{T_\infty=\infty\}}=1$,
i.e., when the chain is non-explosive. 
\item(Monotonicity of the error and of the error bound)  \label{th:fsp_iv}
The approximation error $\norm{p_t-p_t^r}$ and its upper bound $1-p^r_t(\s_r)$ are decreasing in $r$:
\begin{equation}\label{eq:fspermon1}\norm{p_t-p^r_t}\leq \norm{p_t-p^s_t},\quad 1-p^r_t(\s_r)\leq 1-p^s_t(\s_s), \quad \forall s\leq r, \, \forall t\geq0 
\end{equation}
and increasing in $t$:
\begin{equation}\label{eq:fspermon2}\norm{p_t-p^r_t}\geq \norm{p_u-p^r_u} ,\quad  1-p^r_t(\s_r)\geq 1-p^r_u(\s_r), \quad \forall u\leq t, \, \forall r\in\n.
\end{equation}
%
Consequently, the FSP scheme returns not only the approximation $p^r_t$ of $p_t$, but also an approximation $p^r_s$ of $p_s$ for each $s\leq t$ with an error that is bounded uniformly in $s$:
$$\sup_{s\in[0,t]}\norm{p_s-p^r_s}= \norm{p_t-p^r_t}\leq 1-p_t^r(\s_r),\qquad\forall t\geq0.$$
\item(Convergence of bounds).  \label{th:fsp_v}
If $\cup_{r\in\n}\s_r=\s$, then the scheme converges:
$$\lim_{r\to\infty}\norm{p_t-p_t^r}=0,\qquad \forall t\geq0.$$
\end{enumerate}
\end{theorem}

\begin{proof}See Appendix B in the Supplementary Material.\end{proof}

The ideas behind Theorem~\ref{fspthrm} emerge from the following construction. Consider a second chain $X^r$ which is identical to $X$ except that every state outside of the truncation $\s_r$ is turned into an absorbing state. In particular, let $X^r$ be the chain of Lemma~\ref{samechainsh} with $\s_r$ replacing $\cal{D}$, and $Q^r:=(q^r(x,y))_{x,y\in\s}$ replacing $\bar{Q}$, where
\begin{equation}\label{eq:qr}
q^r(x,y):=\left\{\begin{array}{ll}q(x,y)&\text{if }x\in\s_r\\ 0&\text{if }x\not\in\s_r\end{array}\right.
\end{equation}
Lemma \ref{samechainsh} states that the chains $X$ and $X^r$ coincide until (and including) the time $\tau_r$ at which they simultaneously leave the truncation $\s_r$ for the first time, at which point $X^r$ becomes trapped in a state outside of the truncation and never returns to $\s_r$. In contrast, $X$ may return to the truncation, hence the probability $p_t(x)$ that $X$ is at any given state $x$ inside the truncation at time $t$ is greater or equal than the probability that $X^r$ is in the same state at the same time. Since the law of $X^r_t$ (restricted to $\s_r$) is the solution of \eqref{eq:fspode} (Theorem~\ref{CME}), we arrive at Theorem~$\ref{fspthrm} \ref{th:fsp_i}$.

The probability $p^r_t(\s_r)$ that $X^r$ is inside the truncation at time $t$ is the same as the probability $\Pbl{\{t<\tau_r\}}$ that it has not yet left. Theorem~$\ref{fspthrm} \ref{th:fsp_ii}$--$\ref{th:fsp_iii}$ follows from this fact. If $X^r$ has not left the truncation by time $t$, then it has not left the \emph{larger} truncation $\s_{r+1}$ by $t$. Similarly, if the chain has not left $\s_r$ by time $t$, it has not left by any earlier time $s\leq t$. For these reasons, Theorem $\ref{fspthrm} \ref{th:fsp_iv}$ holds.

Due to $\ref{th:fsp_iii}$, proving the convergence of the scheme consists of showing that $\Pbl{\{t<\tau_r\}}$ converges to $\Pbl{\{t<T_\infty\}}$ as $r$ tends to infinity. Recall that $\Pbl{\{t<T_\infty\}}$ is the probability that the chain has not left the state space by time $t$ while $\Pbl{\{t<\tau_r\}}$ is the probability that the chain has not left the truncation $\s_r$ by time $t$. Because the truncations $\s_r$ approach the complete state space as $r$ tends to infinity, it must be the case that $\Pbl{\{t<\tau_r\}}$ approaches $\Pbl{\{t<T_\infty\}}$ or, equivalently, that the scheme converges as stated in Theorem $\ref{fspthrm} \ref{th:fsp_v}$.

The FSP \emph{algorithm} as proposed in \cite{Munsky2006} consists of repeatedly computing $p^r_t$ while increasing the size of the truncation until the error bound $1-p^r_t(\s_r)$ is smaller than some prescribed tolerance. As noted in \cite{MacNamara2007}, the algorithm may not terminate, even if the truncations tend to the state space as $r$ tends to infinity. 
Theorem $\ref{fspthrm} \ref{th:fsp_v}$
clarifies this issue. Although the \emph{scheme} converges (i.e., $p^r_t$ tends to $p_t$ in total variation as $r$ tends to infinity or, equivalently, $p^r_t(\s_r)$ tends to $\Pbl{\{t<T_\infty\}}$), this does not imply that the error bound $1-p_t^r(\s_r)$ converges to zero. This is only the case if the chain is non-explosive (i.e., $\Pbl{\{T_\infty<\infty\}}=0$). Otherwise, $\Pbl{\{t<T_\infty\}}>0$ for all $t>0$ (see the proof of Theorem $\ref{fspthrm} \ref{th:fsp_iii}$) and the algorithm will not terminate if the tolerance is set to be smaller than $1-\Pbl{\{t<T_\infty\}}$. In practice, non-explosivity can be established using a Foster-Lyapunov criterion \cite{Chen1991,Meyn1993b}.

We close this section by pointing out that the FSP scheme can also be used to compute converging approximations of the occupation measure associated with a deterministic time $t$, which tells us how long the chain has spent in state $x$ by time $t$ (see~\cite[Cor.~3.2]{Kuntzthe} for details).

\subsection{The exit time and its associated exit distribution and occupation measure}\label{edom} 
Let $\mu$, $\nu$, $\mu^r$, and $\nu^r$ be defined as in \eqref{eq:edisdef}--\eqref{eq:eoccdef} and \eqref{eq:md2}--\eqref{eq:md22}. Our convention of $\inf\emptyset=\infty$ and the exit time's definition in \eqref{eq:hitthec} 
implies that it is finite if and only if it is strictly less than the explosion time: $\tau(\omega)<\infty \iff \tau(\omega)<T_\infty(\omega)$ for any $\omega\in\Omega$. 
Therefore, $X_\tau$ is defined on $\{\tau<\infty\}$ and $\mu$ is well-defined. Technically, $\mu$, $\nu$, $\mu^r$, and $\nu^r$ are unsigned measures on $([0,\infty)\times\s,\cal{X})$ where $\cal{X}$ is the product sigma algebra of $\tws$ and the Borel sigma $\cal{B}([0,\infty))$ on $[0,\infty)$. When using \eqref{eq:edisdef}--\eqref{eq:eoccdef} and \eqref{eq:md2}--\eqref{eq:md22} to define these four measures, we exploit the fact that $\{[a,b)\times\{x\}:0\leq a<b<\infty,x\in\s\}$ is a $\pi$-system that generates $\cal{X}$. 

From the definition~\eqref{eq:edisdef} of the exit distribution $\mu$, it follows that its mass is the probability that the chain eventually leaves the domain:
\begin{equation}\label{eq:mumass}\mu([0,\infty),\s)
=\Pbl{\{\tau<\infty,X_\tau\in\s\}}=\Pbl{\{\tau<\infty\}},\end{equation}
Similarly, it follows from \eqref{eq:eoccdef} that the mass of the occupation measure is
\begin{equation}\label{eq:numass}\nu([0,\infty),\s)=\Ebl{\int_0^{\tau\wedge T_\infty} \left(\sum_{x\in\s}1_{x}(X_t)\right) dt}=\Ebl{\tau\wedge T_\infty}.\end{equation}
If the chain is non-explosive (i.e., $\Pbl{\{T_\infty=\infty\}}=1$), the mass is the mean exit time. For explosive chains (i.e., $\Pbl{\{T_\infty=\infty\}}<1$), the same holds as long as the chain cannot explode without first exiting the domain (i.e., $\Pbl{\{\tau\leq  T_\infty\}}=1$). 

In~\eqref{eq:edisdef}--\eqref{eq:eoccdef}, we defined the exit distribution $\mu$ and occupation measure $\nu$ probabilistically in terms of the chain $X$. These measures are characterised analytically in terms of the solutions of the ODEs \eqref{eq:jointchar2} in the following theorem.
\begin{theorem}[Analytical characterisation of $\mu$ and $\nu$]\label{charactt} Suppose that \eqref{eq:qhatint} holds. The exit distribution $\mu$ and occupation measure $\nu$ decompose as in \eqref{eq:md1}--\eqref{eq:md11} and their densities $\nu(t,x)$ and $\mu(t,x)$ are non-negative and continuous functions on $[0,\infty)$, for each $x\in\s$. Moreover, 
\begin{equation}\label{eq:jointchar1}\mu(t,x)=1_{\cal{D}^c}(x)\dot{\hat{p}}_t(x),\qquad \nu(t,x)=1_{\cal{D}}(x)\hat{p}_t(x),\qquad\forall x\in\s,\quad  t\in[0,\infty),
\end{equation}
where $\hat{p}_t$ is the minimal non-negative solution (as in Theorem~\ref{CME}) of
\begin{equation}\label{eq:jointchar2}\dot{\hat{p}}_t(x)=\sum_{y\in\cal{D}}\hat{p}_t(y)q(y,x),\qquad \hat{p}_0(x)=\gamma(x),\qquad\forall x\in\s,\quad  t\in[0,\infty).\end{equation}
\end{theorem}
\begin{proof}
See Appendix C in the Supplementary Material.
\end{proof}

The ideas behind the above theorem are similar to those behind Theorem \ref{fspthrm}. In particular, we consider a second chain $\hat{X}$ identical to $X$ except that every state outside of the domain $\cal{D}$ is turned into an absorbing state. That is, let $\hat{X}$ be the chain of Lemma \ref{samechainsh} after replacing $\bar{Q}$ with $\hat{Q}:=(\hat{q}(x,y))_{x,y\in\s}$, where
\begin{equation}\label{eq:qhat}\hat{q}(x,y):=\left\{\begin{array}{ll} q(x,y)&\text{if }x\in\cal{D} \\ 0&\text{if }x\not\in\cal{D} \end{array}\right..\end{equation}
%

The chains $X$ and $\hat{X}$ are identical up until (and including) the time at which they both simultaneously exit the domain via the same state. Therefore the probability $\mu([0,t),x)$ that $X$ has exited the domain by time $t$ via state $x\in\cal{D}^c$ is also the probability that $\hat{X}$ exited via $x$ by time $t$. Because $\hat{X}$ is trapped in the first state it enters once leaving the truncation, it follows that $\mu([0,t),x)$ is the probability that $\hat{X}$ is in state $x$ by time $t$. The characterisation of the exit distribution then follows from Theorem~\ref{CME}. The characterisation of the occupation measure follows similarly. 
The key observation is that once $\hat{X}$ leaves the domain it cannot return, hence the amount of time that $\hat{X}$ spends in a state $x \in \mathcal{D}$ until the moment it exits the domain is the \emph{total} time it will spend in that state.

\subsection*{The marginals} In applications, we are often interested in the distribution of the exit time itself, that is, the \emph{time marginal} of the exit distribution 
\begin{equation}\label{eq:mutc}\mu_T(B):=\mu(B,\s)=\Pbl{\{\tau\in B,X_\tau\in\s\}}=\Pbl{\{\tau\in B\}},\quad \forall B\in\cal{B}([0,\infty)).\end{equation}
Technically, the above is the distribution of $\tau$ restricted to $[0,\infty)$. However, we recover the complete distribution from $\mu_T$ as $\Pbl{\{\tau=\infty\}}=1-\Pbl{\{\tau<\infty\}}=1-\mu_T([0,\infty))$.
Equations \eqref{eq:jointchar1}--\eqref{eq:jointchar2} imply that $\mu(t,x)$ is non-negative and so combining \eqref{eq:md1} and Tonelli's theorem shows that the time-marginal $\mu_T(dt)$ of the exit distribution also has a density $\mu_T(t)$ with respect to the Lebesgue measure and that this density is given by $\mu_T(t)=\sum_{x\in\s}\mu(t,x)$:
\begin{equation}\label{eq:mutden}\mu_T(B)=\gamma(\cal{D}^c)\delta_0(B)+\int_{B}\mu_T(t)dt\qquad\forall B\in\cal{B}([0,\infty)),\quad x\in\s,\end{equation}
where $\delta_0$ denotes the Dirac measure at zero ($\delta_0(B)=1$ if $0\in B$ and $0$ otherwise).

In other cases, we are interested in where on the boundary the exit occurs or where in the domain the chain spends time up until exiting. The \emph{space marginals} of the exit distribution and the occupation measure provide this information:
\begin{align}
\label{eq:musc}\mu_S(x)&:=\mu([0,\infty),x)=\Pbl{\{X_\tau=x,\tau<\infty\}}, &\forall x\in\s,\\
\label{eq:nusc}\nu_S(x)&:=\nu([0,\infty),x)=\Ebl{\int_0^{\tau\wedge T_\infty} 1_{x}(X_t)dt}, &\forall x\in\s.
\end{align}
Clearly, one can obtain explicit expressions for $\mu_T, \mu_S$, and $\nu_S$ in terms of $\hat{p}$. 

\section{Theoretical characterisation of the ETFSP scheme: Proof of Theorem \ref{etfspthrm} and bounding the marginal distributions}\label{etfsp}
We now prove Theorem~\ref{etfspthrm}, which delineates the theoretical properties of the ETFSP scheme. Before delving into the proof, we discuss briefly some the intuitive ideas underlying the proof for the exit distribution (the occupation measure is analogous). 

Consider the auxiliary chain $X^r$ introduced above,
which is identical to the original chain except that each state outside of the truncation $\s_r$ is turned into an absorbing state. Once $X^r$ exits the truncation, it becomes trapped in whichever state it just entered. For this reason, if $X^r$ has not exited the domain by the time it exits $\s_r$, then it will never exit. Theorem \ref{charactt} tells us that  $\mu^r(dt,x)$ is the exit distribution $\rho^r(dt,x)$ of $X^r$ restricted to $[0,t^r_f]\times\s_r$. Thus, Theorem~\ref{etfspthrm}$\ref{th:etfsp_ii}$ follows from the fact that $X$ and $X^r$ are identical up until, and including, the moment that they simultaneously exit the truncation (Lemma \ref{samechainsh}). In contrast with $X^r$, the original chain $X$ may still exit the domain after it leaves the truncation because it does not necessarily get trapped in an absorbing state. During a small interval of time $[t,t+h]$, the probability of exiting the domain $\cal{D}$ via state $x$ is $\mu(t,x)h$ for $X$ and $\rho^r(t,x)h$ for $X^r$. Given that, for any interval size $h$, this probability cannot be greater for $X^r$ than for $X$, the lower bound property in Theorem~\ref{etfspthrm}$\ref{th:etfsp_i}$ follows from the continuity of $\mu(\cdot,x)$ and $\rho^r(\cdot,x)$ (Theorem \ref{charactt}). The remainder of the theorem then follows $\ref{th:etfsp_i}$--$\ref{th:etfsp_ii}$ and the fact that $\tau_r$ is an increasing sequence with limit $T_\infty$ (Lemma \ref{tautin}).

\begin{proof}[Proof of Theorem \ref{etfspthrm}] 
Let $Y^r:=\{Y^r_n\}_{n\in\n}$,  $\{T^r_n\}_{n\in\n}$, and $T^r_\infty$ be the jump chain, jump times, and explosion time of $X^r$.

$\ref{th:etfsp_i}$ Theorem \ref{charactt} tells us that $\nu(t,x)=1_\cal{D}(x)\hat{p}_t(x)$, where $\hat{p}_t$ is the minimal non-negative solution of \eqref{eq:jointchar2}. Theorem \ref{CME} tells us that $\hat{p}_t$ is the time-varying law of the auxiliary chain $\hat{X}$ with rate matrix $\hat{Q}$ defined in \eqref{eq:qhat}. Applying the FSP scheme to $\hat{X}$ instead of $X$ entails solving
\begin{equation}\label{eq:fsphat}\dot{k}^r_t(x)=\sum_{y\in\s_r}k^r_t(y)\hat{q}(y,x)=\sum_{y\in\cal{D}_r}k^r_t(y)\hat{q}(y,x),\qquad k_0^r(x)=\gamma(x),\qquad\forall x\in\s_r\end{equation}
and setting $k^r_t(x)=0$ for all $x\not\in\s_r$. Comparing \eqref{eq:nrdef} and \eqref{eq:fsphat} we can see that 
\begin{equation}\label{eq:knu}\nu^r(t,x)=k^r_t(x),\qquad \forall x\in\cal{D}_r,\quad t\leq t^r_f.\end{equation}
Given \eqref{eq:mrdef2} and the fact that the final times $t^r_f$ are increasing, the second set of inequalities then follows directly from \eqref{eq:jointchar1} and Theorem \ref{fspthrm}$\ref{th:fsp_i}$. Similarly, \eqref{eq:knu}, \eqref{eq:mrdef}, and the finiteness of $\cal{D}_r$ imply that 
$$\mu^{r+1}(t,x)-\mu^r(t,x)=\sum_{z\in\cal{D}_r}(k^{r+1}_t(z)-k^r_t(z))q(z,x)\geq0\qquad\forall x\in\s_r\cap\cal{D}^c,$$
where the inequality follows from Theorem~\ref{fspthrm}$\ref{th:fsp_i}$ and the fact that $x\neq z$ in the above sum so that $q(z,x)\geq 0$. Replacing $\mu^{r+1}$ by $\mu$ and $k^{r+1}$ by $\hat{p}$ in the above argument and applying \eqref{eq:mrdef2} and \eqref{eq:jointchar1} gives us the other set of inequalities.

$\ref{th:etfsp_ii}$ Aside from having to use the fact the explosion time of the chain $X^r$ is a.s. infinite (see \eqref{eq:qrnonexp} in the Supplementary Material), the proof of the expression for the mass of $\nu^r$ is analogous to that for $\mu^r$ and so we skip.  Applying Theorem \ref{charactt} to $X^r$ instead of $X$ shows that $\mu^r$, restricted to $[0,t^r_f]\times\s_r$, coincides with the corresponding restriction of the density of the exit distribution associated with the first time that $X^r$ exits the domain:
$$\tau^r_{\cal{D}}:=\inf\{0\leq t < T^r_\infty:X^r_t\not\in\cal{D}\}.$$
Thus, \eqref{eq:mrdef2} and the definition of the exit distribution \eqref{eq:edisdef} imply that the mass of $\mu^r$ is the probability that $X^r$ exits the domain no later than the final time $t^r_f$ and via a state inside the truncation:
$$\mu^r([0,\infty),\s)=\mu^r([0,t^r_f),\s)=\Pb_\lambda(\{\tau^r_{\cal{D}}\leq t^r_f,X^r_{\tau_D^r}\in\s_r\}).$$
As we now show, this probability is the same as that of the original chain exited the domain no later than the truncation and the final time. The key observation is that whenever $X^r$ leaves the truncation, it becomes trapped in whichever state it just entered. This implies that if $X^r$ has not left the domain by the time it exits the truncation, then it never will. Formally, it follows from Lemma \ref{hitdc} and \eqref{eq:yrsr} in the Supplementary Material that
$$\{\tau_r<\tau_\cal{D}^r\leq t^r_f\}=\{X^r_{\tau_\cal{D}^r}=X^r_{\tau_r},\tau_r<\tau_\cal{D}^r\leq t^r_f\},$$
(recall that Lemma \ref{samechainsh} implies that $X^r$ and $X$ exit the truncation at the same time $\tau_r$). However, the latter set must be the empty set since $X^r_{\tau_{r}(\omega)}(\omega)$ (resp. $X_{\tau^r_\cal{D}(\omega)}(\omega)$) lies inside (resp. outside) of the domain in order for $X^r$ to exit the truncation before it exits the domain ($\tau_{r}(\omega)<\tau^r_{\cal{D}}(\omega)$). Thus, 
$$\{\tau^r_{\cal{D}}\leq t^r_f,X^r_{\tau_D^r}\in\s_r\}=\{\tau^r_{\cal{D}}\leq t^r_f,\tau^r_{\cal{D}}\leq \tau_r,X^r_{\tau_D^r}\in\s_r\}=\{\tau^r_{\cal{D}}\leq t^r_f\wedge\tau_r\}$$
Since \eqref{eq:sameyt} implies $\{\tau^r_{\cal{D}}\leq t^r_f\wedge\tau_r\}=\{\tau\leq t^r_f\wedge\tau_r\}$, the result follows. 

$\ref{th:etfsp_iii}$ As $\ref{th:etfsp_i}$ shows that $\mu-\mu^r$ and $\nu-\nu^r$ are unsigned measures, \eqref{eq:merr}--\eqref{eq:nerr} follow from the fact that the total variation norm of a unsigned measure is its mass, $\ref{th:etfsp_ii}$, the definitions \eqref{eq:nrdef}--\eqref{eq:mrdef2} of $\mu^r$ and $\nu^r$, and the expression for the masses of $\mu$ and $\nu$ in \eqref{eq:mumass}--\eqref{eq:numass}.

$\ref{th:etfsp_iv}$ This follows directly from $\ref{th:etfsp_i}$ and \eqref{eq:merr}--\eqref{eq:nerr}.

$\ref{th:etfsp_v}$ Because Lemmas \ref{tautin} and \ref{hitdc} imply that
$$\lim_{r\to\infty}1_{\{\tau\leq t^r_f\wedge \tau_r\}}=1_{\{\tau<\infty\}},\qquad \lim_{r\to\infty}\tau\wedge t^r_f\wedge \tau_r=\tau\wedge T_\infty, \qquad\text{almost surely},$$
the convergence follows from the monotone convergence theorem and \eqref{eq:merr}--\eqref{eq:nerr}.
\end{proof}

\subsection*{Bounding the marginals} Using the ETFSP scheme we also obtain converging approximations of the marginals $\mu_T$, $\mu_S$, and $\nu_S$ of the exit distribution and occupation measure (see \eqref{eq:mutc}--\eqref{eq:nusc}). In particular, marginalising \eqref{eq:md2}--\eqref{eq:md22}, we obtain approximations of $\mu_T$, $\mu_S$, and $\nu_S$:
\begin{align}\mu^r_T(B):&=\mu^r(B,\cal{S})=\gamma(\cal{D}^c\cap\s_r)\delta_0(B)+\int_B \mu^r_T(t)dt,\qquad &\forall B\in\cal{B}([0,\infty)),\label{eq:tmarap}\\
 \mu^r_S(x):&=\mu^r([0,\infty),x),\qquad \nu^r_S(x):=\nu^r([0,\infty),x),\qquad &\forall x\in\s,\label{eq:smarap}
\end{align}
where $\mu^r_T(t):=\sum_{x\in\s_r}\mu^r(t,x)$. The fact that $\mu^r(dt,x)$ and $\nu^r(dt,x)$ bound from below the exit distribution and occupation measure (Theorem \ref{etfspthrm}~$(i)$) implies that the marginals of the approximations $\mu^r_T(dt)$, $\mu^r_S(x)$, and $\nu^r_S(x)$ bound $\mu_T(dt)$, $\mu_S(x)$, and $\nu_S(x)$ from below. For this reason, the fact that the total variation norm of an unsigned measure is its mass implies that
\begin{equation}\label{eq:marerr}\norm{\mu_T-\mu^r_T}=\norm{\mu_S-\mu^r_S}=\norm{\mu-\mu^r},\qquad\norm{\nu_S-\nu^r_S}=\norm{\nu-\nu^r},\end{equation}
In other words, the errors of the marginal approximations are the same as those of the complete approximations. In full, we have the following corollary of Theorem \ref{etfspthrm}:
\begin{corollary} \label{cormar}Suppose that the premise of Theorem \ref{etfspthrm} is satisfied. Consider the approximations of the marginals $\mu^r_T$, $\mu^r_S$ and $\nu^r_S$ defined in \eqref{eq:tmarap}--\eqref{eq:smarap}.
\begin{enumerate}[label=(\roman*)]
\item(Increasing sequence of lower bounds) 
\label{eq:marginal_i}
The approximations form an increasing sequence of lower bounds:
%
%
\begin{align*}&\mu^0_T(t)\leq \mu^1_T(t)\leq \dots \leq \mu_T(t), &\forall t\in[0,\infty)   \\
&\mu^0_S(x)\leq \mu^1_S(x)\leq \dots \leq \mu_S(x),\qquad \nu^0_S(x)\leq \nu^1_S(x)\leq \dots \leq \nu_S(x),&\forall x\in\s.\end{align*}
\item (Computable error bounds and monotonicity properties) The equalities and inequalities in \eqref{eq:merr}--\eqref{eq:nerrmon} hold identically if we replace $\mu,\mu^r$ with $\mu_T,\mu_T^r$ (or $\mu_S,\mu_S^r$) and $\nu,\nu^r$ with $\nu_S,\nu_S^r$.
\item(Convergence of bounds) Suppose that $\cup_{r}\s_r=\s$ and that $t^r_f\to\infty$ as $r\to\infty$. The approximations of the marginals of the exit distribution converge:
$$\lim_{r\to\infty}\norm{\mu_T-\mu^r_T}=\lim_{r\to\infty}\norm{\mu_S-\mu^r_S}=0.$$
Furthermore, if $\Ebl{\tau\wedge T_\infty}<\infty$, then the approximation of the space marginal of the occupation measure converges:
$$\lim_{r\to\infty}\norm{\nu_S-\nu^r_S}=0.$$
\end{enumerate}
\end{corollary}

\begin{proof}Given \eqref{eq:marerr}, the corollary follows immediately from Theorem \ref{etfspthrm}.
\end{proof}

\section{Applications}\label{applications} 
\begin{figure*}
 \centering
 \includegraphics[width=\textwidth]{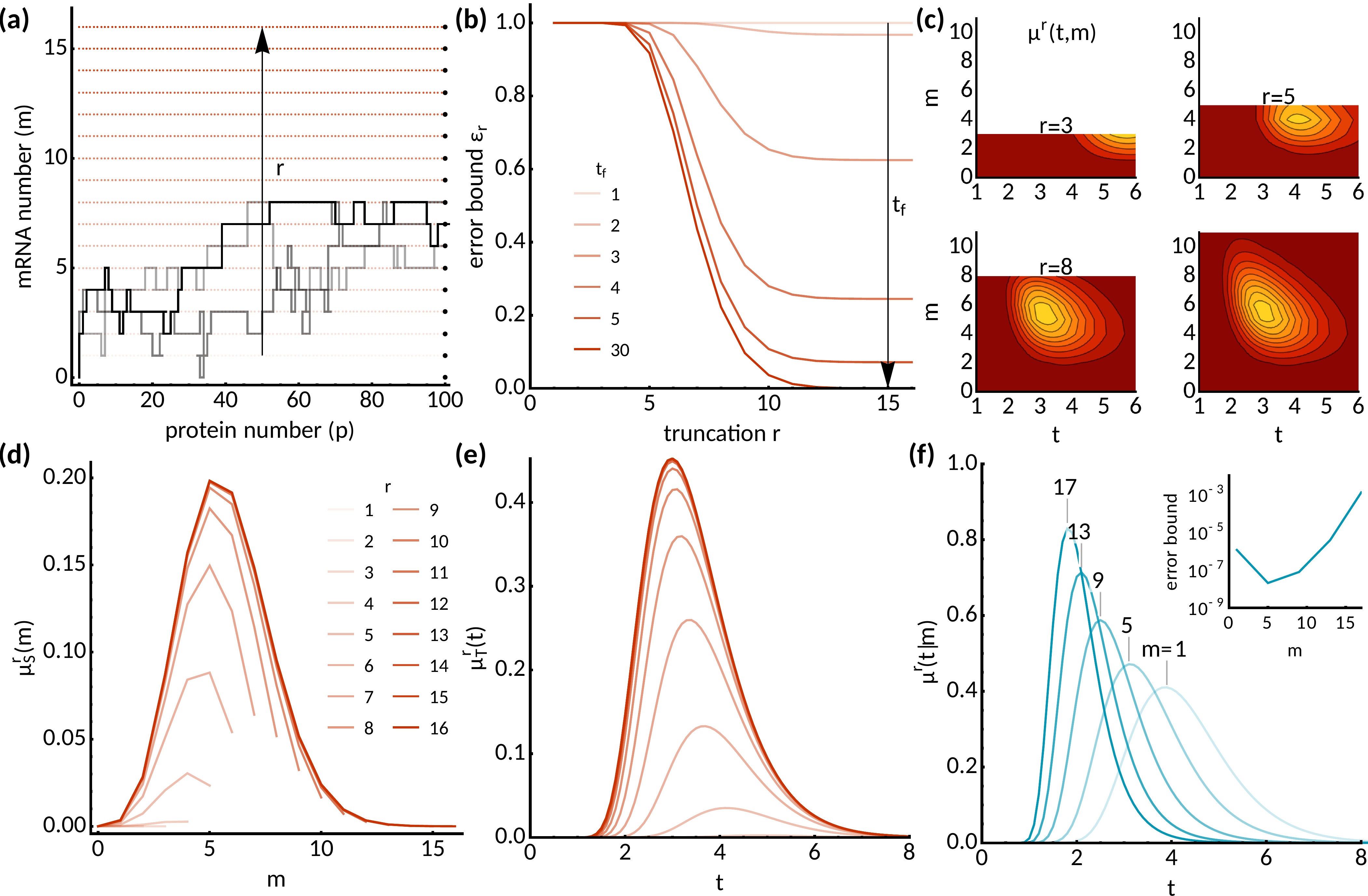}
 \caption{\textbf{Threshold statistics for stochastic gene expression.} 
 \textbf{(a)} Three representative sample paths (light grey, grey, and black lines) of the gene expression model~\eqref{eq:transcription} exiting at $p_c=100$ protein molecules. The red dashed lines indicate the boundaries of the truncations $\mathcal{S}_r$ with increasing $r$.
 \textbf{(b)} The error bound $\varepsilon_r$ decreases with increasing truncation parameter $r$ and final computation time $t_f^r$.
 \textbf{(c)} Lower bounds on the exit distribution for the truncations $r=3,5,8,16$  (yellow indicates maximum probability panel-wise). 
\textbf{(d)} The lower bounds on the mRNA-marginal exit distribution are monotonically increasing with $r$ and become visually indistinguishable for $r>12$. \textbf{(e)} Corresponding bounds on the exit time density. 
 \textbf{(f)} Bounds on the conditional exit time distributions $\mu_T(t|m)$ with $r=20$; inset shows error bounds (\ref{eq:boundgene}). Parameters: $k_1=5$ $k_2=1$, $k_3=10$, $k_4=0.1$ and initial condition $\gamma(x)=1_0(m)1_0(p)$ for all panels.}
 \label{fig:gene}
\end{figure*}
In this section, we apply the ETFSP scheme to two biological examples from the literature. 
To simplify the exposition, we assume without loss of generality that the chain starts inside the domain: $\Pbl{\{X_0\in\cal{D}\}}=\gamma(\cal{D})=1$.

\subsection{Threshold model for stochastic gene expression}

Proteins perform essential functions inside living cells. These molecules are expressed from genes through a series of biochemical reactions, 
and their absolute levels (and the timings in which these are reached) are critical to cell decisions, such as differentiation \cite{Dandach2010} or lysis in the bacteriophage $\lambda$ \cite{singh2014}. Let us consider a simple model of gene expression involving the transcription and degradation of mRNA molecules (with rates $k_1$ and $k_2$, respectively), the synthesis of a protein from each mRNA molecule (with rate $k_3$), and the degradation of proteins (with rate $k_4$):
\begin{align}
\label{eq:transcription}
\varnothing \xrightarrow[]{k_1} \text{mRNA} \xrightarrow[]{k_2} \varnothing, \qquad
 \text{mRNA} \xrightarrow[]{k_3} \text{mRNA} + \text{Protein},\qquad   \text{Protein} \xrightarrow[]{k_4} \varnothing.
 \end{align}
The state of the system is $x=(m,p)$, where $m$ is the number of mRNAs and $p$ is the number of proteins; hence the state space is $\s =\n^2$. The reactions obey mass-action kinetics and the rate matrix is given by
$$q\left((m_1,p_1),(m_2,p_2)\right) = \left\{\begin{array}{ll}-k_1-k_2m_1-k_3m_1-k_4p_1 &\text{if }(m_2,p_2)=(m_1,p_1)\\k_1&\text{if }(m_2,p_2)=(m_1+1,p_1)\\ k_2m_1&\text{if }(m_2,p_2)=(m_1-1,p_1)\\k_3m_1&\text{if }(m_2,p_2)=(m_1,p_1+1)\\ k_4p_1&\text{if }(m_2,p_2)=(m_1,p_1-1)\\0&\text{otherwise}\end{array}\right..$$

We are interested in characterising the time taken for the protein number $p$ to attain a critical level $p_c$. To this end, we consider the domain
\begin{align*}
 \mathcal{D} := \{(m,p)\in\n^2:p<p_c\},
\end{align*}
so that the exit time $\tau$ from the domain $\mathcal{D}$ correspond to the first instant at which $p_c$ proteins accumulate.
We compute the lower bounds $\mu^r(t,(m,p_c))$
of the exit distribution 
$\mu(t,(m,p_c))$, the joint distribution of the exit time and the number of mRNAs present at exit. For ease of notation, in the rest of this section, we omit the protein number argument (as it is $p_c$ at time $\tau$), and we write $(m,p_c)$ as $m$. 

We use the truncations
\begin{align*}
 \mathcal{S}_r=\{(m,p)\in \n^2:p \leq p_c, \, m < r\}\qquad\forall r\in\n,
\end{align*}
shown in Fig.~\ref{fig:gene}(a).  
Fig.~\ref{fig:gene}(b) shows how the error bound $\varepsilon_r$ 
decreases to zero with $r$ and $t^r_f$,
whereas Fig.~\ref{fig:gene}(c) shows the lower bounds $\mu^r(t,m)$ for various values of the truncation parameter $r$ and $t^r_f=30$ (with $\varepsilon_r<10^{-4}$ for $r=16$).

The exit time correlates negatively with the level of mRNA: the more mRNA molecules are present, the higher the expression, and the quicker the protein number rises. Figs.~\ref{fig:gene}(d) and (e) show the corresponding lower bounds $\mu^r_S(m)$ and $\mu^r_T(t)$ on the space and time marginals $\mu_S(m)$ and $\mu_T(t)$, respectively.

To gain a quantitative understanding of the anti-correlation between the exit time and mRNA numbers, we also compute the density of the exit time conditioned on the number of mRNA present:
\begin{align}\label{eq:condgene}
  \mu(t|m) := \frac{\mu(t,m)}{\mu_S(m)} \ge \frac{\mu^r(t,m)}{\mu^r_S(m)+\varepsilon_r}=:\mu^r(t|m).
\end{align}
The bound on the right-hand side follows from the fact that $\mu_S(m)$ is no greater than $\mu^r_S(m)+\varepsilon_r$ due to the definition of the total variation norm. 
Integrating both sides of~\eqref{eq:condgene}, we obtain the following bound on the total variation distance between the conditional density and its approximation: 
\begin{align}
 \label{eq:boundgene}
 \norm{\mu(\cdot|m)-\mu^r(\cdot|m)}\leq \frac{\varepsilon_r}{\mu_S^r(m)+\varepsilon_r}.
\end{align}
Fig.~\ref{fig:gene}(f) shows this density computed using $r=20$ for various values of $m$. As expected, the mode of the distribution decreases with increasing mRNA number but, interestingly, the density also narrows with increasing $m$. In the inset of Fig.~\ref{fig:gene}(f) we verify that the approximation error is small for each $m$.

\subsection{Fixation statistics in population dynamics}
The ETFSP framework can be used to provide insights into the fixation (or extinction) statistics of competing populations with small numbers. Common models in ecology and evolution are of the Lotka-Volterra type \cite{constable2017}. Let us consider the population dynamics of two competing species $S_1$ and $S_2$: 
\begin{align}
\label{eq:fixation_stochastics}
 S_i \xrightarrow[]{b_i/K} 2 S_i, \qquad  S_i \xrightarrow[]{d_i/K} \varnothing, \qquad  S_i + S_j \xrightarrow[]{c_{ij}/K^2} S_j,\qquad\forall i,j\in\{1,2\},
\end{align}
with state space $\s:=\n^2$. The first and second reactions describe the birth and death of individuals with rates $b_i>0$ and $d_i>0$, respectively. The third reaction describes intra- and inter-species competition of strength $c_{ij}>0$. The parameter $K>0$ is the effective carrying capacity. Let us denote the numbers of individuals by $x=(x_1,x_2)$. The rate matrix is given by 
$$q(x,y)=\left\{\begin{array}{ll}-\sum_{i=1}^2 \left(w^+_i(x)+w^-_i(x)+w_{i1}(x)+w_{i2}\right)&\text{if }y=x\\ w^+_i(x)&\text{if }y=x+e_i\enskip \forall i\in\{1,2\}\\ w^-_i(x)+w_{i1}(x)+w_{i2}&\text{if }y=x-e_i\enskip \forall i\in\{1,2\}\\0&\text{otherwise}\end{array}\right.,$$
where $e_1:=(1,0)$, $e_2:=(0,1)$ and
$$  w_{i}^+(x) := b_i \frac{x_i}{K},\qquad  w_{i}^-(x) := d_i \frac{x_i}{K} ,\qquad w_{ij}(x):= c_{ij} \frac{x_ix_j}{K^2},\qquad\forall i,j\in\{1,2\}.$$
For simplicity, we fix $c_{11}=c_{12}=c_{21}=c_{22}=1$ and $K=30$. 

\paragraph{Deterministic dynamics}

\begin{figure*}
 \centering
 \includegraphics[width=0.9\textwidth]{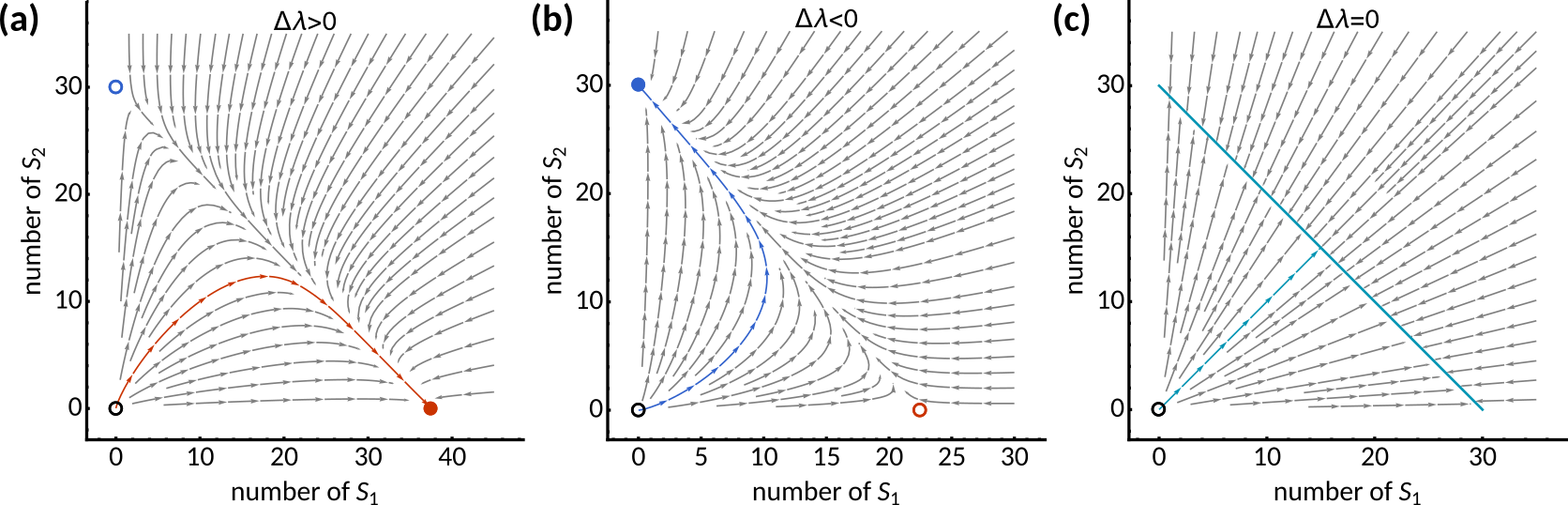}
 \caption{\textbf{Dynamics of the deterministic fixation model.} Phase portraits of the ODE model of two competing species~\eqref{eq:deterministic_pop} for different values of the growth rate difference $\Delta \lambda $. 
 \textbf{(a,b)} If $\Delta\lambda \neq 0$, the system has a single stable fixed point corresponding to the fixation of the species with the highest growth rate. Filled dots denote stable fixed points; open circles denote unstable fixed points. \textbf{(c)} For equal growth rates ($\Delta\lambda=0$), the dynamics approaches a line of fixed points, representing coexistence of the two species (neutral case).}
 \label{fig:lv1}
\end{figure*}

The deterministic dynamics of the populations is modelled with the set of ODEs:
\begin{align}
\label{eq:deterministic_pop}
{\dot x_i} = w_i^+(x)-w_i^-(x) - w_{i1}(x)-w_{i2}(x) \qquad\forall i=1,2.
\end{align}
The equilibrium $(x_1,x_2)=(0,0)$ representing the extinction of both populations is unstable. Fixation of $S_1$ occurs when $S_2$ goes extinct and vice versa, i.e., when the dynamics approaches one of the two axes. Which of the two species becomes extinct depends on the growth rate difference:
\begin{align}
\Delta\lambda:=\lambda_1-\lambda_2 \quad \text{where} \quad
\lambda_i=b_i-d_i, \qquad \forall i,j\in\{1,2\}.
\label{eq:net_growth}
\end{align}
For $\Delta\lambda>0$ (Fig.~\ref{fig:lv1}(a)), there is an unstable fixed point on the $x_1=0$ axis and a stable one on the $x_2=0$ axis; hence the trajectories approach the stable fixed point leading to fixation of $S_1$. For $\Delta\lambda<0$, 
the situation is reversed resulting in the fixation of $S_2$ (Fig.~\ref{fig:lv1}(b)). For equal growth rates ($\Delta\lambda=0$), the dynamics approaches an invariant manifold (a line of fixed points) on which the two species coexist with ratios depending on their initial populations (Fig.~\ref{fig:lv1}(c)).

\paragraph{Computation of fixation probabilities and times}

\begin{figure*}
 \centering
 \includegraphics[width=0.9\textwidth]{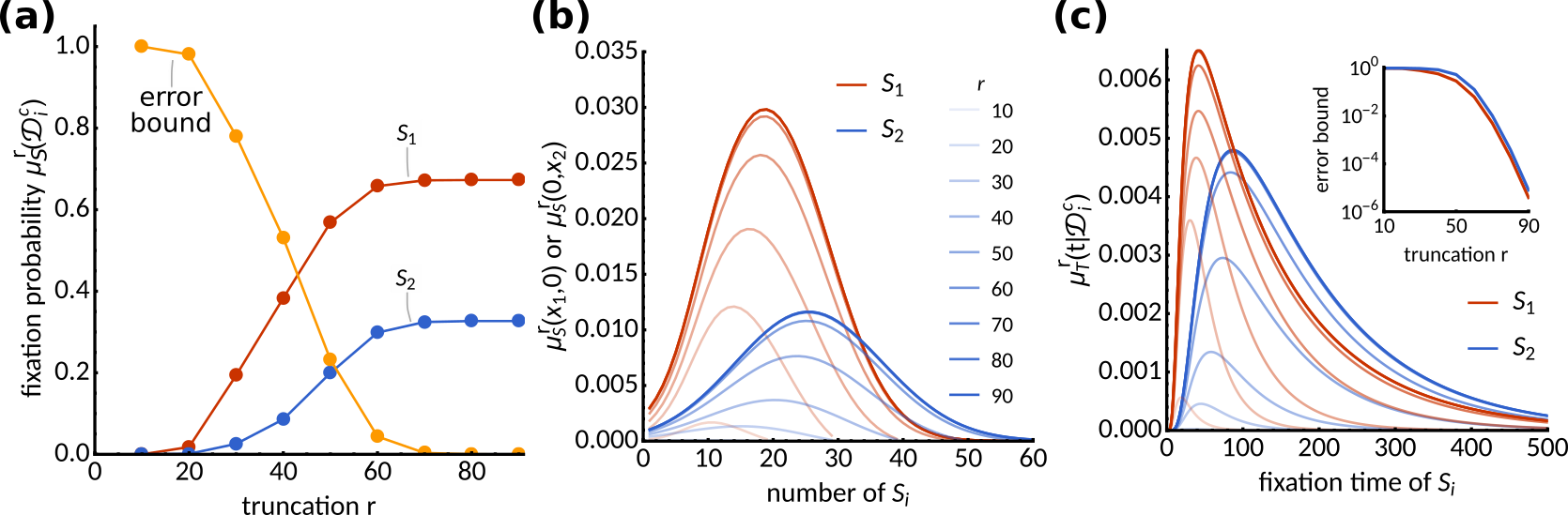}
 \caption{\textbf{Computation of fixation probabilities and times.} 
 \textbf{(a)} The lower bounds on the fixation probabilities converge (red $S_1$, blue $S_2$) with increasing truncation parameter $r$, and the error bound (yellow) approaches zero. \textbf{(b)} The lower bounds on the exit location distributions converge with increasing $r$ (from light to dark). \textbf{(c)} The lower bounds on the density of fixation times also converge. Parameters: $b_1 = 2$, $b_2 = 5$, $d_1=1$, $d_2=4$, $\Delta\lambda=0$.
 }
 \label{fig:lv2}
\end{figure*}

In the stochastic setting, both species $S_1$ and $S_2$ have non-zero probability of becoming fixed regardless of the value of $\Delta\lambda$.
To study this phenomenon, we consider the exit time from the domain $$\mathcal{D} = \{ (x_1,x_2)\in\n^2 \, | \, x_1>0, x_2>0 \},$$ with complement $\mathcal{D}^c$ that can be decomposed into the disjoint subsets
$$\mathcal{D}^c_1:=\{ (x_1,0)\in\n^2:x_1>0 \}, \qquad  \mathcal{D}^c_2=\{ (0,x_2)\in\n^2:x_2>0 \},$$
representing, respectively,  the fixation of $S_1$ and of $S_2$,
and a third subset $\{(0,0)\}$ 
representing the extinction of both species. 

We compute lower bounds $\mu^r_S(\cal{D}^c_1)$ and $\mu^r_S(\cal{D}^c_2)$ 
on the fixation probabilities 
using the ETFSP scheme 
and the truncations
$$\mathcal{S}_r=\{(x_1,x_2)\in \n^2: x_1+x_2\le r\},$$
with final computation time $t^r_f:=3000$, and initial condition $\gamma=1_{(10,10)}$.

In Fig.~\ref{fig:lv2}(a), the results for the neutral case $\Delta\lambda=0$ with different death rates $d_1<d_2$ show that the error bound $\varepsilon_r$ decreases with $r$ and can be made arbitrarily small. 
However, in contrast with the deterministic case, 
$S_1$ fixes with higher probability and 
the fixation dynamics does not depend only on the growth rate difference $\Delta \lambda$ but also on the difference in death rates $\Delta d := d_1 - d_2$.  This demographic noise 
drives the species with the higher death rate ($S_2$) to extinction more frequently~\cite{constable2016}.

To study this effect, we consider the distribution of $S_1$ or $S_2$ upon fixation (Fig.~\ref{fig:lv2}(b)). The probabilities of exiting either through the states $x_1$ in $\mathcal{D}^c_1$ or the states $x_2$ in $\mathcal{D}^c_2$ are bounded by
\begin{align*}
\mu_S(x_1,0) \ge \mu^r_S(x_1,0),\qquad \mu_S(0,x_2) \ge \mu^r_S(0,x_2).
\end{align*}
As shown in Corollary~\ref{cormar}, the bounds $\mu^r_S(x_1,0)$ and $\mu^r_S(0,x_2)$ increase monotonically in $r$ and converge---in our numerics, the approximations are visually indistinguishable for $r>60$. Note that the exit location distributions are wide and not clearly peaked around the intersections of $\mathcal{D}_1^c$ and $\mathcal{D}_2^c$ with the deterministic manifold. Indeed, the $S_1$-exit location distribution peaks at smaller values than deterministically plausible due to higher demographic noise along the direction of $S_2$ disturbing the dynamics away from the deterministic stable manifold.

To characterise the time at which either fixation occurs, we compute bounds on the fixation time densities. The fixation time of $S_1$ (resp. $S_2$) is the exit time conditioned on $S_1$ (resp. $S_2$) fixing and its density is given by
\begin{equation}\label{eq:fixtimes}
 \mu_T(t|\cal{D}^c_i):=\frac{\mu(t,\mathcal{D}^c_i)}{\mu_S(\mathcal{D}^c_i)} \ge \frac{\mu^r(t,\mathcal{D}^c_i)}{\mu^r_S(\mathcal{D}^c_i)+\varepsilon_r}=:\mu_T^r(t|\cal{D}^c_i)\qquad \forall i=1,2.
 \end{equation}
Fig.~\ref{fig:lv2}(c) shows that the bounds on the conditional densities are monotonically increasing, whereas the inset shows that  
the bound of the approximation error
$$\norm{\mu_T(\cdot|\cal{D}^c_i)-\mu_T^r(\cdot|\cal{D}^c_i)}\leq \frac{\varepsilon_r}{\mu^r_S(\mathcal{D}^c_i)+\varepsilon_r}.$$
decreases with $r$.

\paragraph{The effects of demographic noise}

\begin{figure*}
 \centering
 \includegraphics[width=0.9\textwidth]{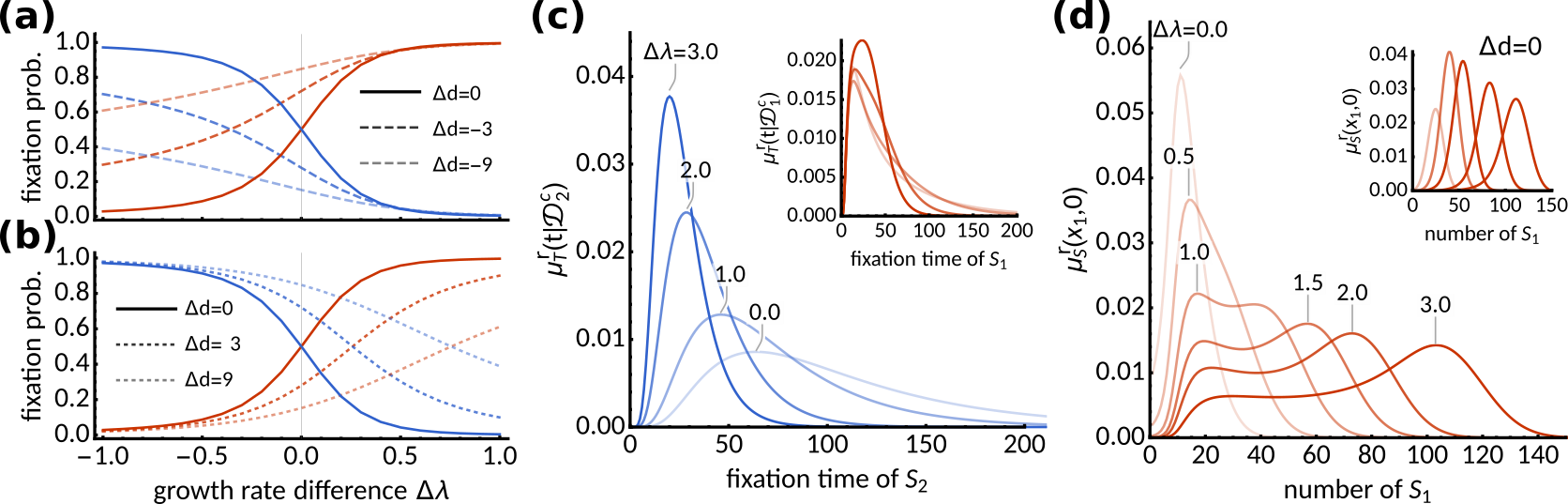}
 \caption{\textbf{Fixation in the presence of demographic noise.} \textbf{(a,b)} Fixation probabilities as a function of the growth difference $\Delta\lambda$ for various $\Delta d$ values with: $d_1=1$ in (a); $d_2=1$ in (b)). Red lines, $S_1$; blue lines, $S_2$. \textbf{(c)} The density of $S_2$ fixation times for different $\Delta \lambda$ values with $d_1=1, d_2=10$. Inset shows the corresponding densities of $S_1$. \textbf{(d)} Distribution of $S_1$ individuals at fixation with $d_1=1, d_2=10$ and varying $\Delta\lambda$. Inset shows $d_1=d_2=1$ case (no demographic noise). All computations carried out with $r=200$ ($20301$ states) ensuring that $\varepsilon_{200}<10^{-8}$. Birth rates: $b_1= 1+\Delta\lambda+d_1$ and $b_2=1+d_2$.}
 \label{fig:lv3}
\end{figure*}

Using ETFSP with a large truncation, we investigate how the behaviour of the model depends on the growth rate difference $\Delta\lambda$ and the death rate difference $\Delta d$, a measure of demographic noise. 
In the absence of demographic noise ($\Delta d=0$), as in the deterministic case, the fixation of $S_1$ is favoured if $\Delta \lambda>0$, and the converse is true if $\Delta \lambda<0$ (Fig.~\ref{fig:lv3}(a)). However, an increase in the demographic noise of $S_2$ ($\Delta d<0$) leads to a higher fixation probability of $S_1$. If the demographic noise is large enough, the fixation of $S_1$ becomes favoured even if $\Delta\lambda<0$. Conversely, Fig.~\ref{fig:lv3}(b) shows that increasing the demographic noise of $S_1$ ($\Delta d>0$) favours fixation of $S_2$ over $S_1$ even if $\Delta \lambda>0$. 

Next, we focus on the case $\Delta \lambda>0, \, \Delta d<0$ where the most likely outcome is consistent with the deterministic case (i.e., $S_1$ is more likely to fix). Fig.~\ref{fig:lv3}(c) shows the density of fixation times computed using \eqref{eq:fixtimes} for $S_1$ and $S_2$ as a function of $\Delta\lambda$. The density of fixation times becomes narrower with increasing $\Delta \lambda$ (inset) indicating that large growth rate differences attenuate the stochasticity. Despite the fixation of $S_2$ being less likely than that of $S_1$, the time required for this event decreases with $\Delta\lambda$: the mode of the conditional distribution $\mu_T(t|\cal{D}^c_2)$ shifts to smaller times.

Fig.~\ref{fig:lv3}(d) shows that the distribution of $S_1$ individuals at fixation is bimodal for moderate values of $\Delta\lambda$ and remains broad for larger values. This is the result of strong demographic fluctuations in the direction of $S_2$ such that fixation of $S_1$ can occur at small population size. 
If no demographic noise is present ($\Delta d=0$), the distributions are unimodal (inset) and considerably narrower regardless of the value of $\Delta \lambda$. 
In summary, demographic noise significantly alters the dynamics of small populations and can even reverse the direction of fixation predicted by deterministic models.

\section{Discussion}\label{conclu} 
In this paper, we have introduced and characterised the ETFSP scheme, which yields converging approximations of the exit distribution and occupation measure associated with the exit from a domain of continuous-time Markov chains. The ETFSP scheme consists of solving the system of coupled linear ODEs \eqref{eq:nrdef}--\eqref{eq:mrdef} and yields approximations of the desired measures.
The total variation distance between the exit distribution and its approximation is bounded by one minus the mass of the approximation. Hence the quality of the approximation can be evaluated with no extra effort than that required for its computation.

We have considered minimal chains, i.e., those that do not explode or those that are killed off after exploding. 
A distinction arises for non-minimal chains, which are re-initialised after exploding~\cite{Chung1967,Freedman1983,Rogers2000a}. In this case, the FSP and ETFSP still yield monotonically increasing lower bounds on the relevant measures and the computable error bounds hold identically. However, they do not converge to the measures associated with non-minimal chains but to those associated with minimal chains (Theorems \ref{etfspthrm} and \ref{fspthrm}). The details pertinent to non-minimal chains are left as future work. Although we have not discussed time-inhomogeneous chains relevant in some applications \cite{Voliotis2016,Dattani2017}, we anticipate that both the ETFSP and FSP schemes apply identically when the rate matrix $Q$ is replaced with its time-inhomogeneous analogue.

There are several issues worth considering for the application of the ETFSP scheme. Chief among them is the fact that the number of states often grows quickly with the desired accuracy resulting in large systems of ODEs. Resource-efficient implementations of the FSP scheme have been developed to tackle this issue and can be adapted to the ETFSP setting (see \cite{Dinh2016} and references therein). 
To do so, notice that \eqref{eq:nrdef} is the set of ODEs obtained by applying the FSP scheme to \eqref{eq:laweqs} with $\cal{D}$ replacing $\s$. In other words, $\nu^r(t,\cdot)$ is $\gamma_{\cal{D}_r}\exp(tQ_{\cal{D}_r})$, where 
$\gamma_{\cal{D}_r}$ and $Q_{\cal{D}_r}$ are restrictions 
to the truncated domain $\cal{D}_r$. 
The corresponding approximation of the exit distribution is then obtained by rewriting \eqref{eq:mrdef} as
\begin{equation}\label{eq:mudefalt}\mu^r(t,x)=\sum_{y\in\cal{D}_r}\nu^r(t,y)q(y,x),\qquad\forall x\in\s_r\cap\cal{D}^c.\end{equation}
The rapid growth in the number of states can also be mitigated by guiding the truncation choice using simulation-based criteria \cite{Munsky2007,Sidje2015}, moment bounds and Markov's inequality to obtain a priori error bounds~\cite{Kuntzthe,Kuntz2017}, or other state space exploration techniques (see \cite{DeSouzaeSilva1992,Dinh2016} and references therein). For cases where there are too many important states for ETFSP to handle, Galerkin methods \cite{engblom2009,Engblom2009a,Jahnke2010} could be adapted to the exit time setting using Theorem \ref{charactt}.

Solving \eqref{eq:nrdef}--\eqref{eq:mrdef} numerically introduces an additional source of error~\cite{Moler2003}. A simple way to control this error is to apply randomisation techniques \cite{Gross1984,Dinh2016} on \eqref{eq:mrdef} to obtain lower bounds $\nu^r$, and using \eqref{eq:mudefalt} to compute lower bounds on $\mu^r$. The error bounds in Theorem~$\ref{etfspthrm} \ref{th:etfsp_iii}$ hold if $\mu^r$ and $\nu^r$ are replaced with their lower bounds.

Numerically solving the ODEs \eqref{eq:nrdef} to obtain $\nu^r$ and performing the matrix-vector multiplication in \eqref{eq:mudefalt} often leads to an accumulation of errors in $\mu^r$. We circumvented this issue using an adaptive ODE solver~\cite{Hindmarsh2005} to solve the joint system \eqref{eq:nrdef}--\eqref{eq:mrdef}, hence ensuring that the errors of both $\nu^r$ and $\mu^r$ are taken into account by the solver. A promising alternative here is to apply Krylov methods of the type in \cite{Burrage2006} to this joint system of ODEs.

Lastly, we did not address how to bound the approximation error of the occupation measure in practice. As shown in Theorem \ref{etfspthrm}, the approximation error \eqref{eq:nerr} depends on $\Ebl{\tau\wedge T_\infty}$, which is bounded from above by 
the mean exit time $\Ebl{\tau}$.  For a broad class of chains (those with `rational rate matrices'), the mean exit time can itself be bounded using linear or semidefinite programming approaches \cite{Helmes2001,Kuntzthe}. For more general chains, one can employ Foster-Lyapunov criteria~\cite{Menshikov2014,Kuntzthe}.

In summary, the ETFSP computes converging approximations of the exit distribution and occupation measure with controlled errors. As demonstrated in Section~\ref{applications}, such highly accurate approximations can provide valuable insights into the dynamics of biochemical networks and interacting populations. Although our examples were biological, computing these measures is important to other fields, for instance, to quantify customer waiting times \cite{Melamed1984,Melamed1984a}, modelling computer-communication and transaction processing systems \cite{Bernardo2007}, computing reliability measures of complex systems \cite{Grassmann2000}, or in model checking \cite{Milios2017}. 

\vspace{10pt}
\noindent\textbf{Acknowledgements:} We thank the two anonymous referees for their helpful remarks that have significantly improved this manuscript. J.K. gratefully thanks Prof. Sophia Yaliraki for an important stint in her research group during which the material presented in this manuscript was partially developed.

\bibliographystyle{siamplain} 
\bibliography{etfspbib}

\appendix

\begin{center}\Large{\textbf{Supplementary Material}}\end{center}

\section{The Gillespie Algorithm and the proof of Lemma \ref{samechainsh}}\label{preproofs}
Given an \emph{initial condition} $Z$, we construct our Markov chain $X$ recursively by running Algorithm \ref{gilalg} below commonly known as the \emph{Gillespie Algorithm} or the \emph{stochastic simulation algorithm}. The name of the algorithm itself stems from \cite{Gillespie1976} and its origins trace back to \cite{Feller1940,Kendall1950}. In particular, the algorithm constructs the \emph{jump times} $\{T_n\}_{n\in\n}$ at which transitions occur and the jump chain $Y:=\{Y_n\}_{n\in\n}$.

\begin{algorithm}[h]
\begin{algorithmic}[1]
\STATE{$Y_0:=Z$, $T_0:=0$}
\FOR{$n=1,2,\dots$}
\STATE{sample $U_{n}\sim\operatorname{uni}((0,1))$ independently of $\{Z,\xi_1,\dots,\xi_{n-1},U_1,\dots,U_{n-1}\}$}
\STATE{sample $\xi_{n}\sim\operatorname{exp}(1)$ independently of $\{Z,\xi_1,\dots,\xi_{n-1},U_1,\dots,U_{n}\}$}
\IF{$q(Y_{n-1},Y_{n-1})\neq0$}
\STATE{$T_n:=T_{n-1}-\xi_n/q(Y_{n-1},Y_{n-1})$}
\ELSE
\STATE{$T_n:=T_{n-1}+\xi_n$}
\ENDIF
\STATE{$i:=0$}
\WHILE{$U_n>\sum_{j=0}^i\pi(Y_{n-1},x_j)$}
\STATE{$i:=i+1$}
\ENDWHILE
\STATE{$Y_n:=x_i$}
\ENDFOR
 \end{algorithmic}
 \caption{The Gillespie Algorithm on $\s=\{x_1,x_2,x_3\dots\}$}\label{gilalg}
\end{algorithm}

In this paper, we fix an underlying measurable space $(\Omega,\cal{F})$ on which $Z$, $\xi_1$, $\xi_2$, $\dots$, $U_1$, $U_2$, $\dots$ appearing in Algorithm \ref{gilalg} are defined and a probability measure $\Pb$ on $(\Omega,\cal{F})$ such that, under $\Pb$, the initial condition $Z$ has law $\gamma$, the random variable $U_n$ is uniformly distributed on $(0,1)$ for each $n\in\zp$, the random variable $\xi_n$ is exponentially distributed with unit mean for each $n\in\zp$, and the random variables $Z$, $\xi_1$, $\xi_2$, $\dots$, $U_1$, $U_2$, $\dots$ are independent. Formally, such a construction can be carried out using Theorems 12.2 and 26.1 in \cite{Rogers2000a}.

\begin{proof}[Proof of Lemma \ref{samechainsh}] Let $\bar{\Pi}$ denote the one-step matrix obtained by replacing $Q$ with $\bar{Q}$ in \eqref{eq:jumpmatrix}. To construct $\bar{X}$ we run Algorithm \ref{gilalg} employing the same $Z$, $\{\xi_n\}_{n=1}^\infty$, and $\{U_n\}_{n=1}^\infty$ as for $X$ but with $\bar{Q}$ and $\bar{\Pi}$ replacing $Q$ and $\Pi$ to obtain the chain's jump times $\{\bar{T}_{n}\}_{n\in\n}$ and jump chain $\bar{Y}:=\{\bar{Y}_{n}\}_{n\in\n}$ and then we apply \eqref{eq:cpathdef} with $\{\bar{T}_{n}\}_{n\in\n}$ and $\bar{Y}$ replacing $\{\bar{T}_{n}\}_{n\in\n}$ and $Y$. 
Because the rate matrices coincide on $\cal{D}$, \eqref{eq:jumpmatrix} implies that the jump matrices also coincide on $\cal{D}$:
$$\pi(x,y)=\bar{\pi}(x,y)\qquad \forall x\in\cal{D},\enskip y\in\s.$$
Algorithm \ref{gilalg} and the above imply that  
\begin{equation}\label{eq:nja}\bar{Y}_{n+1}(\omega)=Y_{n+1}(\omega)\text{ for all }\omega\in\Omega\text{ such that }\bar{Y}_n(\omega)=Y_n(\omega)\in\cal{D}.\end{equation}
Due to the definition of the exit times of the jump chains, we have that
\begin{align*}\sigma&=\infty\cdot 1_{\{Y_0\in\cal{D},Y_1\in\cal{D},\dots\}}+\sum_{k=1}^\infty k1_{\{Y_0\in\cal{D},\dots,Y_{k-1}\in\cal{D},Y_k\not\in\cal{D}\}},\\
 \bar{\sigma}&=\infty\cdot 1_{\{\bar{Y}_0\in\cal{D},\bar{Y}_1\in\cal{D},\dots\}}+\sum_{k=1}^\infty k1_{\{\bar{Y}_0\in\cal{D},\dots,\bar{Y}_{k-1}\in\cal{D},\bar{Y}_k\not\in\cal{D}\}}. \end{align*}
Because $Y_0 =Z= \bar{Y}_0$, combining the above expression with \eqref{eq:nja} tells us that $\sigma(\omega)=\bar{\sigma}(\omega)$ for each $\omega\in\Omega$. Since $\sigma(\omega)\geq k$ only if
$$Y_0(\omega)\in\cal{D},\quad Y_k(\omega)\in\cal{D},\quad\dots,\quad Y_{k-1}(\omega)\in\cal{D},$$
the first equation in \eqref{eq:sameyt} also follows from \eqref{eq:nja}. Using once again the fact that the rate matrices coincide on $\cal{D}$ and the definition of the jump times in Algorithm \ref{gilalg}, the second equation in \eqref{eq:sameyt} follows from the first. Lemma \ref{hitdc} then implies the second and third equations in \eqref{eq:sametime}. Putting \eqref{eq:sametime}, \eqref{eq:sameyt}, and the definition of the chains in \eqref{eq:cpathdef} together we obtain \eqref{eq:sameproc}.
\end{proof}
\section{The proof of the theoretical properties of the FSP scheme}\label{detocc} 
In the following proof, let  $X^r$ be the auxiliary chain introduced immediately after Theorem \ref{fspthrm} and $Y^r:=\{Y^r_n\}_{n\in\n}$,  $\{T^r_n\}_{n\in\n}$, and $T^r_\infty$ be its jump chain, jump times, and explosion time.

\begin{proof}[Proof of Theorem \ref{fspthrm}] Substituting $Q$ with $Q^r$ in \eqref{eq:laweqs} and comparing with \eqref{eq:fspode}, it follows that 
\begin{equation}\label{eq:xfsp}p^r_t(x)=\Pbl{\{X_t^r=x,t<T^r_\infty\}},\qquad\forall x\in\s_r.\end{equation}

$\ref{th:fsp_i}$ Because $Q$ and $Q^r$ coincide on $\s_r$, Lemma \ref{samechainsh} tells us that both $X$ and $X^r$ leave for the first time $\s_r$ at the same moment (namely, $\tau_r$). Similarly, the time of exit from $\s_r$ for the jump chains $Y$ and $Y^r$ coincides and we denote it by $\sigma_r$. Replacing $Q$ by $Q^r$ in \eqref{eq:jumpmatrix}, we see that the one-step matrix $\Pi^r:=(\pi^r(x,y))_{x,y\in\s}$ is such that $\pi^r(x,\cdot)=1_x(\cdot)$ for each $x\not\in\s_r$. For this reason, Algorithm \ref{gilalg} implies that for any $\omega\in\Omega$
$$Y^r_n(\omega)=x\not\in\s_r\Rightarrow Y^r_{n+m}(\omega)=x\quad\forall m\in\n.$$
Due to the definition of $\sigma_r$, we have that $Y^r_{\sigma_r(\omega)}(\omega)$ does not belong to $\s_r$ if $\sigma_r(\omega)$ is finite  and so
\begin{equation}\label{eq:yrsr}Y^r_n(\omega)=Y^r_{\sigma_{r}(\omega)}(\omega)\not\in\s_r\qquad \forall n\geq\sigma_{r}(\omega),\quad\text{if}\quad\sigma_{r}(\omega)<\infty,\end{equation}
formalising the notion that $X^r$ gets stuck in the first state in enters once it leaves the truncation. The above implies that $\{Y^r_n=x\}=\{Y^r_n=x,n<\sigma_{r}\}$ for every $x\in\cal{S}_r$. Using the above,
\begin{align}&\{Y^r_n=x,T_n^r\leq t<T_{n+1}^{r}\}=\{Y^r_n=x,T_n^r\leq t<T_{n+1}^{r},n<\sigma_r\}\nonumber\\
&=\{Y^{r+1}_n=x,T_n^{r+1}\leq t<T_{n+1}^{r+1},n<\sigma_r\}\subseteq\{Y^{r+1}_n=x,T_n^{r+1}\leq t<T_{n+1}^{r+1}\},\nonumber\end{align}
for all $x\in\s_r$, where the second equality follows from \eqref{eq:sameyt} in Lemma \ref{samechainsh} after noting that  the definition of $Q^r$ in \eqref{eq:qr} remains unchanged if we replace $q(x,y)$ with $q^{r+1}(x,y)$. Taking the union over $n\in\n$, we obtain 
\begin{align*}
\{X_t^r=x,t<T^r_\infty\}&=\bigcup_{n=0}^\infty \{Y_n^r=x,T^r_n\leq t<T^r_{n+1}\}\\
&\subseteq\bigcup_{n=0}^\infty \{Y_n^{r+1}=x,T^{r+1}_n\leq t<T^{r+1}_{n+1}\}=\{X_t^{r+1}=x,t<T^{r+1}_\infty\},
\end{align*}
for all $x\in\s_r$. Taking expectations and applying Theorem \ref{CME} yields $p_t^r(x)\leq p^{r+1}_t(x)$ for each $x\in \s_r$. Replacing $X^{r+1}$ with $X$ in this argument, shows that $p_t^r(x)\leq \dots\leq p_t(x)$ for each $x\in\s_r$.

$\ref{th:fsp_ii}$ Theorem \ref{charactt} and \eqref{eq:md1} tell us that 
$$\Pbl{\{\tau_r\leq t\}}=\sum_{x\not\in\s_r}\left(\gamma(x)+\int_0^t\left(\sum_{y\in\s_r}p_s^r(y)q(y,x)\right)ds\right).$$
Theorem \ref{CME} and \eqref{eq:xfsp} then imply that
$$\Pbl{\{\tau_r\leq t\}}=\Pbl{\{X^r_t\not\in\s_r,t<T_\infty\}}=\Pbl{\{t< T_\infty^r\}}-\Pbl{\{X^r_t\in\s_r,t<T_\infty\}}.$$
Because $\s_r$ is finite, \eqref{eq:qr} implies that $x\mapsto q^r_x$ is a bounded function. Using the definition of $T^r_\infty$ and the law of large numbers we have that $X^r$ is non-explosive:
\begin{align}\label{eq:qrnonexp}T^r_\infty&=\sum_{n=0}^\infty\left(1_{\{q^r_{Y^r_n}=0\}}-\frac{1_{\{q^r(Y^r_n,Y^r_n)\neq0\}}}{q^r(Y^r_n,Y^r_n)}\right)\xi_{n+1}\\
&\geq\left(1\wedge\min_{x\in\s_r}\frac{1}{-q(x,x)}\right)\sum_{n=0}^\infty\xi_{n+1}=\infty,\quad \text{almost surely}.\nonumber\end{align}
For this reason, using \eqref{eq:xfsp} we have that
$$p_t^r(\s)=p_t^r(\s_r)=\Pbl{\{X^r_t\in\s_r,t<T_\infty\}}=1-\Pbl{\{\tau_r\leq t\}}=\Pbl{\{t<\tau_r\}}.$$

$\ref{th:fsp_iii}$ The equality and inequality follow from $\ref{th:fsp_i}$--$\ref{th:fsp_ii}$ and the fact that the total variation norm of an unsigned measure is its mass. The function $t\mapsto p_t(\s)$ is non-increasing because $p_t(\s)=\Pbl{\{t<T_\infty\}}$. Theorem 5 of Chapter II.18 in \cite{Chung1967} implies that $p_t(\s)<1$ for a given $t>0$ if and only if $p_t(\s)<1$ for all $t>0$. These facts and the monotone convergence theorem imply that $p_t(\s)<1$ for any given $t>0$ if and only if $\Pbl{\{T_\infty=\infty\}}<1$. For this reason, the inequality is sharp if and only if $\Pbl{\{T_\infty=\infty\}}=1$.

$\ref{th:fsp_iv}$ This is an immediate consequence of $\ref{th:fsp_ii}$--$\ref{th:fsp_iii}$ and the fact that $\{\tau_r\}_{r\in\n}$ is an increasing sequence.

$\ref{th:fsp_v}$  The monotone convergence theorem and $\ref{th:fsp_iii}$ imply that
$$\lim_{r\to\infty}\norm{p_t-p^r_t}_{TV}=\Pbl{\{t<T_\infty\}}-\lim_{r\to\infty} \Pb(\{t<\tau_{r}\}).$$
The claim then follows from Lemma \ref{tautin}.
\end{proof}
\section{The proof of the analytical characterisation of the exit distribution and occupation measure}
The proof of Theorem \ref{charactt} relies on the auxiliary chain $\hat{X}$ defined immediately after the theorem's statement. In what follows, let $\hat{Y}:=\{\hat{Y}_n\}_{n\in\n}$, $\{\hat{T}_n\}_{n\in\n}$, and $\hat{T}_\infty$ to denote the jump chain, jump times, and explosion time of $\hat{X}$. The theorem's proof  builds on the following simple lemma.

\begin{lemma}\label{exbst}  The chain $X$ does not explode before first leaving the domain if and only if $\hat{X}$ does not explode:
$$1_{\{\tau\leq T_\infty\}}=1_{\{\hat{T}_\infty=\infty\}}\qquad \Pb\text{-almost surely.}$$
%
\end{lemma}
\begin{proof}By its definition \eqref{eq:hitthec}, the exit time is no greater than the explosion time if and only if the chain exits the domain before any explosion occurs or the chain neither exits the domain nor explodes:
$$\{\tau\leq T_\infty\}=\{\tau<T_\infty\}\cup\{\tau=T_\infty=\infty\}=\{\tau<\infty\}\cup\{\tau=T_\infty=\infty\}.$$
Because these events are disjoint, it is enough to argue that
\begin{equation}\label{eq:1nfhdsf}1_{\{\tau<\infty\}}=1_{\{\tau<\infty,\hat{T}_\infty=\infty\}},\quad 1_{\{\tau=T_\infty=\infty\}}=1_{\{\tau=\hat{T}_\infty=\infty\}},\quad \Pb\text{-almost surely.}\end{equation}
Because the jump times of both chains agree as long as no exit occurs, see \eqref{eq:sameyt}, we have that 
$$T_\infty(\omega)=\lim_{k\to\infty}T_k(\omega)=\lim_{k\to\infty}\hat{T}_k(\omega)=\hat{T}_\infty(\omega)\qquad\forall\omega\in\Omega:\sigma(\omega)=\infty.$$
where $\sigma$ denotes the time-step \eqref{eq:sigma} that $Y$ and $\hat{Y}$ simultaneously (Lemma \ref{samechainsh}) leave the domain. The second equation in \eqref{eq:1nfhdsf} then follows from Lemma \ref{hitdc}. To prove the first equation, notice that an analogous argument as that behind \eqref{eq:yrsr} shows that 
\begin{equation}\label{eq:absjump}\hat{Y}_k(\omega)=\hat{Y}_{\sigma(\omega)}(\omega)\quad\forall k\geq\sigma(\omega),\quad\text{if}\quad\sigma(\omega)<\infty.\end{equation}
Combining the above with the law of large numbers, we have that
\begin{align*}1_{\{\sigma=l\}}\hat{T}_\infty&=1_{\{\sigma=l\}}\sum_{k=0}^\infty (\hat{T}_{k+1}-\hat{T}_k)\geq 1_{\{\sigma=l\}}\sum_{k=l}^\infty (\hat{T}_{k+1}-\hat{T}_k)\\
&\geq 1_{\{\sigma=l\}}\sum_{k=l}^\infty \xi_{k+1}=1_{\{\sigma=l\}}\cdot\infty\quad\Pb\text{-almost surely},\end{align*}
for any $l\in\n$. Summing the above over $l\in\n$, we find that
$$1_{\{\sigma<\infty\}}\hat{T}_\infty\geq1_{\{\sigma<\infty\}}\cdot\infty\qquad\Pb\text{-almost surely}.$$
The first equation in \eqref{eq:1nfhdsf} then follows from Lemma \ref{hitdc}.
\end{proof}

We are now in a position prove our characterisation of $\mu$ and $\nu$.
\begin{proof}[Proof of Theorem 2.5] We begin with the occupation measure.  Lemma \ref{samechainsh} implies that
\begin{align}\label{eq:fnhyduafns}\int_0^{t\wedge\tau\wedge T_\infty}1_x(X_s)ds&=\int_0^{t\wedge\tau\wedge \hat{T}_\infty}1_x(\hat{X}_s)ds\\
&=1_{\{\tau\leq t\}}\int_0^{\tau\wedge\hat{T}_\infty}1_x(\hat{X}_s)ds+1_{\{\tau>t\}}\int_0^{t\wedge\hat{T}_\infty}1_x(\hat{X}_s)ds.\nonumber\end{align}
If we can argue that
\begin{equation}\label{eq:h7dha}1_{\{\tau\leq t\}}\int_\tau^{t\wedge \hat{T}_\infty}1_x(\hat{X}_s)ds=0\qquad\forall x\in\cal{D},\end{equation}
then adding the left-hand side of \eqref{eq:h7dha} to the right-hand side of \eqref{eq:fnhyduafns}, taking expectations, using Tonelli's theorem, and applying Theorem \ref{CME} to $\hat{X}$ yields the characterisation of the occupation measure. The above follows from the fact that $\hat{X}$ hits an absorbing state as soon as it leaves the domain. Formally, Lemmas \ref{hitdc}--\ref{samechainsh} imply that
$$\{\tau\leq s,\hat{T}_k\leq s<\hat{T}_{k+1}\}=\{\hat{T}_{\max\{\sigma,k\}}\leq s<\hat{T}_{k+1}\}=\{\sigma\leq k, \hat{T}_k\leq s<\hat{T}_{k+1}\}$$
for all $k\in\n$. The above and \eqref{eq:absjump} tell us that
\begin{align}\nonumber \hat{X}_s(\omega)&=\sum_{k=0}^\infty 1_{\{\hat{T}_k\leq s< \hat{T}_{k+1}\}}(\omega)\hat{Y}_k(\omega)=\sum_{k=\sigma(\omega)}^\infty 1_{\{\hat{T}_k\leq s< \hat{T}_{k+1}\}}(\omega)\hat{Y}_k(\omega)\\
&=1_{\{s\geq \hat{T}_\sigma\}}(\omega)\hat{Y}_{\sigma(\omega)}=\hat{Y}_{\sigma(\omega)}(\omega)=X_{\tau(\omega)}(\omega),\qquad\forall (s,\omega):\tau(\omega)\leq s<\hat{T}_\infty(\omega),\label{eq:stuck}\end{align}
where the last equality follows from Lemmas \ref{hitdc}--\ref{samechainsh}. By definition, the paths of $X$ are c\'adl\'ag (with respect to the discrete topology on $\s$) implying that $X_\tau$ lies outside of the domain and \eqref{eq:h7dha} follows from \eqref{eq:stuck}.

For the characterisation of the exit distribution, fix any $x\not\in\cal{D}$ and note that
\begin{align*}\mu([0,t],x)&=\Pb(\{X_{\tau}=x,\tau\leq t\})=\Pb(\{X_{\tau}=x,\tau\leq t, \tau\leq T_\infty\})\\
&=\Pb(\{{X}_{\tau}=x,\tau\leq t,\hat{T}_\infty=\infty\})=\Pb(\{\hat{X}_{\tau}=x,\tau\leq t< \hat{T}_\infty,\hat{T}_\infty=\infty\})\\
&=\Pb(\{\hat{X}_{t}=x,\tau\leq t< \hat{T}_\infty,\hat{T}_\infty=\infty\})=\Pb(\{\hat{X}_{t}=x,\tau\leq t< \hat{T}_\infty\})\\
&=\Pb(\{\hat{X}_{t}=x,t< \hat{T}_\infty\})=\hat{p}_t(x).
\end{align*}
The first equality follows from the definition of $\mu$, the second that of $\tau$, the third from Lemma \ref{exbst}, the fourth from Lemma \ref{samechainsh}, the fifth from \eqref{eq:stuck}, the sixth from Lemma \ref{exbst}, the seventh from the fact that $\hat{X}_t$ lies outside of the domain only if its exit time is no greater than $t$ and \eqref{eq:sametime}, and the eighth from Theorem \ref{CME}.  Exploiting the continuity of $t\mapsto\hat{p}_t(x)$ (Theorem \ref{CME}) and applying the monotone convergence theorem to $\mu(\cdot,x)$ implies that
$$\mu([0,t),x)=\lim_{n\to\infty}\mu([0,t(1-1/n)],x)=\lim_{n\to\infty}\hat{p}_{t(1-1/n)}(x)=\hat{p}_t(x),\qquad\forall x\not\in\cal{D},$$
thus completing the proof of the first equation in \eqref{eq:jointchar1}. 
\end{proof}

\end{document}